%% file: Arhiv_Version.tex
\documentclass{article}

\usepackage{amssymb}
\usepackage{amssymb}
\usepackage{amsmath,amsthm}
\usepackage{enumerate}
\usepackage{mathtools}
\usepackage{graphicx}
\usepackage{textcomp}
\usepackage{authblk}

\input{macros}

\newtheorem{definition}{Definition}
\newtheorem{proposition}{Proposition}
\newtheorem{lemma}{Lemma}
\newtheorem{example}{Example}
\newtheorem{remark}{Remark}

\begin{document}




\title{$2$-coherent and $2$-convex Conditional Lower Previsions}


\author[1]{Renato Pelessoni\thanks{renato.pelessoni@econ.units.it}}
\author[1]{Paolo Vicig\thanks{paolo.vicig@econ.units.it}}
\affil[1]{DEAMS ``B. de Finetti''\\
	University of Trieste\\
	Piazzale Europa~1\\
	I-34127 Trieste\\
	Italy}

\renewcommand\Authands{ and }





\maketitle

\begin{abstract}
In this paper we explore relaxations of (Williams) coherent and convex conditional previsions that form the families of $n$-coherent and $n$-convex conditional previsions, at the varying of $n$.
We investigate which such previsions are the most general one may reasonably consider, suggesting (centered) $2$-convex or, if positive homogeneity and conjugacy is needed, $2$-coherent lower previsions. Basic properties of these previsions are studied. In particular,
we prove that they satisfy the 
Generalized Bayes Rule and always have a $2$-convex or, respectively, $2$-coherent natural extension.
The role of these extensions is analogous to that of the natural extension for  coherent lower previsions.
On the contrary, $n$-convex and $n$-coherent previsions with $n\geq 3$ either are convex or coherent themselves or have no extension of the same type on large enough sets. 
Among the uncertainty concepts that can be modelled by $2$-convexity, we discuss generalizations of capacities and niveloids to a conditional framework
and show that the well-known risk measure Value-at-Risk only guarantees to be centered $2$-convex.
In the final part, we determine the rationality requirements of $2$-convexity and $2$-coherence from a desirability perspective,
emphasising how they weaken those of (Williams) coherence.

\smallskip
\noindent \textbf{Keywords.}
Williams coherence, $2$-coherent previsions, $2$-convex previsions, Generalised Bayes Rule.
\end{abstract}


\section*{Acknowledgement}
*NOTICE: This is the authors' version of a work that was accepted for publication in the International Journal of Approximate Reasoning. Changes resulting from the publishing process, such as peer review, editing, corrections, structural formatting, and other quality control mechanisms may not be reflected in this document. Changes may have been made to this work since it was submitted for publication. A definitive version was subsequently published in the International Journal of Approximate Reasoning, vol. 77, October~2016, pages 66--86, doi:10.1016/j.ijar.2016.06.003 $\copyright$ Copyright Elsevier

http://www.sciencedirect.com/science/article/pii/S0888613X16300792.

\vspace{0.3cm}
$\copyright$ 2016. This manuscript version is made available under the CC-BY-NC-ND 4.0 license http://creativecommons.org/licenses/by-nc-nd/4.0/

\begin{center}
	\includegraphics[width=2cm]{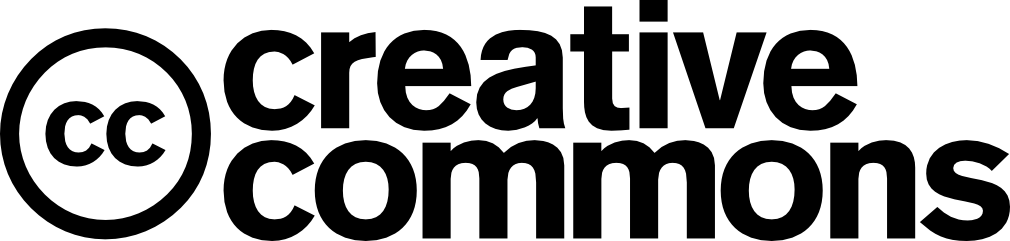}
	\includegraphics[width=2cm]{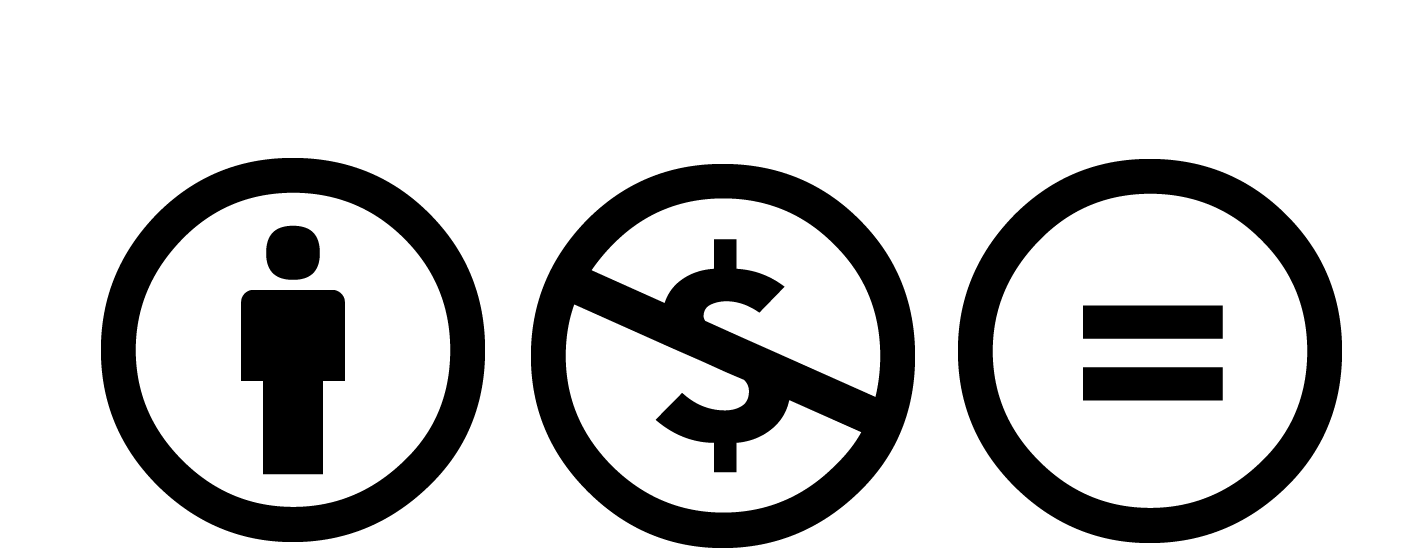}
\end{center}

\section{Introduction}
\label{sec:introduction} In his influential book 
\emph{Statistical Reasoning with Imprecise Probabilities} \cite{wal91}, P.
Walley developed a behavioural approach to \emph{imprecise} probabilities
(and previsions) extending de Finetti's \cite{def74} interpretation
of coherent \emph{precise} previsions. Operationally, this was
achieved through a relaxation of de Finetti's betting scheme.

In fact, following de Finetti, $P$ is a coherent precise prevision
on a set $\qset$ of gambles if and only if for all $m$,
$n\in\natsetp$, $s_1,\ldots,s_m,r_1,\ldots,r_n\geq 0$,
$X_1,\ldots,X_m$, $Y_1,\ldots,Y_n\in\qset$, defining
$G=\sum_{i=1}^{m}s_i(X_i-P(X_i))-\sum_{j=1}^{n}r_j(Y_j-P(Y_j))$, it
holds that $\sup G\geq 0$. The terms $s_i(X_i-P(X_i))$,
$-r_j(Y_j-P(Y_j))$ are proportional (with coefficients or \emph{stakes}
$s_i$, $r_j$) to the \emph{gains} arising from, respectively,
buying $X_i$ at $P(X_i)$ or selling $Y_j$ at $P(Y_j)$.
A coherent lower prevision $\lpr$ on $\qset$ may be defined
in a similar way, just restricting $n$ to belong to $\{0,1\}$. This means
that the betting scheme is modified to allow selling at most one
gamble.
Several other betting scheme variants have been investigated in the
literature, either extending coherence for lower previsions
(conditional lower previsions) or weakening it (previsions that are
convex, or avoid sure loss). In particular, a convex lower prevision
is defined introducing a convexity constraint $n = 1, \sum_{i=1}^{m}
s_i = r_1=1$ in the betting scheme. In \cite[Appendix B]{wal91}
$n$-coherent previsions are studied, as a different relaxation of
coherence.

In this paper, we explore further variations of the behavioural
approach /betting scheme: $n$-coherent and $n$-convex conditional
lower previsions, formally defined later on as generalisations of
the $n$-coherent (unconditional) previsions in \cite{wal91}. Our
major aims are:
\begin{itemize}
\item[a)]
to explore the flexibility of the behavioural approach and its capability to encompass different uncertainty models;
\item[b)]
to point out which are the basic axioms/properties of coherence
which hold even for much looser consistency concepts.
\end{itemize}
Referring to b) and with a view towards the utmost generality, we
shall mainly concentrate on the extreme quantitative models that can
be incorporated into a (modified) behavioural approach. This does
not imply that these models should be regarded as preferable to
coherent lower previsions. On the contrary they will not, as far as
certain questions are concerned. For instance, inferences will
typically be rather vague. However, it is interesting and somehow
surprising to detect that certain properties like the Generalised
Bayes Rule must hold even for such models, or that they can be
approached in terms of desirability.

$N$-coherence and $n$-convexity may be naturally seen as relaxations
of, respectively, (Williams) coherence and convexity. These and
other preliminary concepts are recalled in Section
\ref{sec:preliminaries}. Starting from the weakest reasonably sound
consistency concepts, we explore $2$-convex
lower previsions in Section~\ref{sec:2-convex_previsions}.
In Section \ref{sec:basic_2_convex} we characterise them
by means of axioms, on a special set of
conditional gambles generalising a linear space and termed $\dlin$
(Definition~\ref{def:def_dlin}, Proposition~\ref{2-convexity_structure}). Interestingly, it turns out that
$n$-convexity with $n\geq 3$ and convexity are equivalent on
$\dlin$. $2$-convex previsions display some drawbacks:
in Section \ref{2_convex_natural_extension},
it is shown that a $2$-convex natural extension may be defined
and its properties are discussed,
but its finiteness is not guaranteed.
Moreover, as detailed in Section \ref{sec:drawbacks_2_convexity},
the property of internality may fail
(with some limitations,
for instance lack of internality cannot be two-sided);
agreement with conditional implication (the Goodman-Nguyen
relation) is not guaranteed either.
In Section \ref{sec:centered_2-convex}, we show that the
special subset of centered $2$-convex previsions is not affected by
these problems. In Section \ref{sec:2_coherent}, $2$-coherent lower
previsions are discussed and characterised on $\dlin$ (Proposition
\ref{pro:conditional 2-convex prevision on linear space}).
We compare $2$-coherence and $n$-coherence in Section \ref{sec:2_n_coherence}:
again, $n$-coherence ($n\geq 3$) and coherence are equivalent on $\dlin$.
$N$-coherent previsions ($n\geq 3$) defined on a generic set of gambles $\mathcal{S}$ have no
$n$-coherent extension on sufficiently large supersets whenever the
equivalence does not hold already on $\mathcal{S}$. We show also that $2$-coherence should be
preferred to $2$-convexity when positive homogeneity and conjugacy
are required.
The $2$-coherent natural extension is introduced and studied
in Section \ref{sec:2_coherent_natural_extension}.
$2$-coherent lower previsions always have it.
The extent of the Generalised Bayes Rule for $2$-coherent lower previsions
is discussed in Section \ref{subsec:GBR}.
Models that can be accommodated into the
framework of $2$-convexity or $2$-coherence, but not of coherence, are presented in
Section \ref{sec:weakly_consistent_models}.
We focus on how $2$-convexity can motivate defining conditional versions of capacities and niveloids,
and on the consistency properties of Value-at-Risk,
a well-known risk measure which is centered $2$-convex,
but may even fail to be $2$-coherent.
In Section \ref{sec:weak_consistency} we analyse
$2$-convexity and $2$-coherence in a desirability
approach. Generalising prior work by Williams \cite{wil75, wil07}
for coherence, we focus on the correspondence between these
previsions and sets of desirable gambles, and on establishing the
ensuing desirability rules.
The major differences with the rules for Williams coherence are pointed out
in the comments following
Propositions \ref{pro:lower_from_set} and \ref{pro:set_from_lowerc}.
Section~\ref{sec:conclusions}
concludes the paper. An earlier presentation of the topics in this paper,
less extended and without proofs, was delivered at the ISIPTA'15 Symposium
\cite{pel15}.

\section{Preliminaries}
\label{sec:preliminaries}
The starting points for our investigation are the known consistency concepts of coherent and convex lower conditional prevision \cite{pel05, pel09, wil75, wil07}.
They both refer to an arbitrary (non-empty) set $\dset$ of conditional gambles,
that is of conditional bounded random variables.
We denote by $X|B$ a generic conditional gamble,
where $X$ is a gamble and $B$ is a non-impossible event ($B\neq\varnothing$).
It is understood here that $X:\prt\rightarrow\rset$ is defined on an underlying partition $\prt$
of atomic events $\omega$,
and that $B$ belongs to the powerset of $\prt$.
Therefore,
any $\omega\in\prt$ implies either $B$ or its negation $\nega{B}$
(in words, knowing that $\omega$ is true determines the truth value of $B$,
i.e. $B$ is known to be either true or false).
Given $B$,
the conditional partition $\prt|B$ is formed by the conditional events $\omega|B$,
such that $\omega$ implies $B$ (implies that $B$ is true) and $X|B:\prt|B\rightarrow\rset$
is such that $X|B(\omega|B)=X(\omega)$, $\forall \omega|B\in\prt|B$.
Because of this equality,
several computations regarding $X|B$ can be performed by means of the restriction of $X$ on $B$.
In particular,
it is useful for the sequel to recall that $\sup(X|B)=\sup_{B}X=\sup\{X(\omega):\omega\in\prt, \omega\Rightarrow B\}$,
and $\inf(X|B)=\inf_{B}X=\inf\{X(\omega):\omega\in\prt, \omega\Rightarrow B\}$.

As a special case, letting $\Omega$ be the sure event,
we have that $X|\Omega=X$ is an unconditional gamble.
Further,
$A|B$ is a conditional event if $A$ is an event
(or its indicator $I_{A}$ - we shall generally employ the same notation $A$ for both).

As customary,
without further qualifications,
a lower prevision $\lpr$ is
a map from $\dset$ into the real line,
$\lpr:\dset\rightarrow\reals$.
However,
a lower prevision is often interpreted as a supremum buying price \cite{wal91}.
For instance,
if a subject assigns $\lpr(X|B)$ to $X|B$,
he is willing to buy $X$,
conditional on $B$ occurring,
at any price lower than $\lpr(X|B)$.
Referring to this behavioural interpretation,
the following Definitions \ref{def:coh_conv_prev}, \ref{def:2-convexity}, \ref{def:cond_2-coherence} require different degrees of consistency for $\lpr$,
according to whether certain gains depending on $\lpr$ avoid losses bounded away from $0$.
They differ as to the buying and selling constraints they impose.
\vspace{0.1cm}
\begin{definition}
\label{def:coh_conv_prev}
\begin{itemize}
Let $\LP:\dset\rightarrow\Rset$ be given.
\item[a)]
$\LP$ is a \emph{coherent} conditional lower prevision on $\dset$ iff,
for all $m\in\natsetp$,
$\forall X_0|B_0,\ldots,X_m|B_m\in\dset$, $\forall s_0,\ldots,s_m\geq 0$,
defining
$\Ss=\bigvee\{B_{i}:s_{i}\neq 0, i=0,\ldots,m\}$ and
$\LG=\sum_{i=1}^{m}s_{i}B_{i}(X_{i}-\LP(X_{i}|B_{i}))-s_{0}B_{0}(X_{0}-\LP(X_{0}|B_{0}))$,
it holds,
whenever $\Ss\neq\varnothing$,
that $\sup\{\LG|\Ss\}\geq 0$.

\item[b)]
$\LP$ is a \emph{convex} conditional lower prevision on $\dset$ iff, for all $m\in\natset^{+}$,
$\forall X_0|B_0,\ldots,X_m|B_m\in\dset$, $\forall s_1,\ldots,s_m\geq 0$
such that
$\sum_{i=1}^{m}s_i=1$ \emph{(convexity constraint)},
defining $\LGC=\sum_{i=1}^{m}s_{i}B_{i}(X_{i}-\LP(X_{i}|B_{i}))-B_{0}(X_{0}-\LP(X_{0}|B_{0}))$,
 $\Ss=\bigvee\{B_{i}:s_{i}\neq 0, i=1,\ldots,m\}$,
it is $\sup\{\LGC|\Ss\vee B_{0}\}\geq 0$.

\item[b1)]
$\LP$ is centered convex or \emph{C-convex} on $\dset$ iff
it is convex and, $\forall X|B\in\dset$, it is $0|B\in\dset$ and $\lpr(0|B)=0$.
\end{itemize}
\end{definition}
In the behavioural interpretation recalled above,
Definition \ref{def:coh_conv_prev}a) considers buying at most $m$ conditional gambles
$X_1|B_1,\ldots,X_m|B_m$ (also no one, when $m=0$)
at prices $\lpr(X_1|B_1),\ldots,\lpr(X_m|B_m)$, respectively,
and selling at most one gamble $X_0|B_0$ at its supremum buying price $\lpr(X_0|B_0)$.
The gain $\LG$
is a linear combination with stakes $s_0,\ldots,s_m$ of the $m+1$ gains
from these transactions.
It is conditioned on $\Ss$,
to rule out both trivial transactions ($\LG=0$, since $s_0=\ldots=s_m=0$) and
the case that $\LG=0$
because no transaction takes place
(when $B_0,\ldots,B_m$ are all false).
Then, coherence requires the non-negativity of the supremum of $\LG$,
conditional on at least one non-trivial transaction being effective.
The interpretation of Definition \ref{def:coh_conv_prev}b) is similar:
what changes is the convexity constraint on the stakes
$(s_0=1), s_1, \ldots, s_m$.
This implies that $\LGC$ is the gain from one selling transaction and at least
one buying transaction.

The definition of coherent conditional lower prevision is a structure free version of Williams coherence,
discussed in \cite{pel09}.
It is more general than Walley's coherence \cite{wal91},
in particular it is not necessarily conglomerable and always allows for a \emph{natural extension},
i.e. there exists an extension on any set of a Williams coherent assessment that
is Williams coherent too and least committal.
The notion of convex lower prevision is still more general,
and was introduced in \cite{pel05},
extending the unconditional convexity studied in \cite{pel03}.
Convex previsions can incorporate various uncertainty models,
including convex risk measures,
non-normalised possibility measures,
and others.
However,
the special subclass of C-convex lower previsions guarantees
better consistency properties.
Among these,
there always exists a convex natural extension of these measures,
whose properties are analogous to those of the natural extension \cite[Theorem 9]{pel05}.

Even though coherent and convex lower previsions can be defined on any set of conditional gambles,
they are characterised by a few axioms on the special environment $\dlin$ defined next.
\begin{definition}
\label{def:def_dlin}
Let $\xset$ be a linear space of gambles and $\bset\subset\xset$ a set of  (indicators of) events in $\xset$.
Suppose that $\Omega\in\bset$ and that $\xset$ is \emph{stable by restriction},
i.e. $BX\in\xset, \forall B\in\bset, \forall X\in\xset$.
Setting $\bsetp=\bset-\{\varnothing\}$,
define
\begin{eqnarray}
\label{eq:def_dlin}
\dlin=\{X|B: X\in\xset, B\in\bsetp\}.
\end{eqnarray}
\end{definition}

Note that,
since $\bset\subset\xset$,
the condition $\Omega\in\bset$ implies $1\in\xset$ and, therefore,
$\xset$ contains all real constants.

The sets $\dlin$ may be viewed as generalisations to conditional gambles of
linear spaces of unconditional gambles,
to which they reduce when
$\bset=\{\Omega,\varnothing\}$.
Not surprisingly
then, characterisations on $\dlin$ have an unconditional counterpart
on linear spaces.

\begin{proposition}
\label{pro:conditional convex prevision on linear space}
Let $\LP:\dlin\rightarrow\Rset$ be a conditional lower prevision.
\begin{itemize}
\item[a)] $\lpr$ is coherent on $\dlin$ if and only if \cite{wil07}
\begin{itemize}
\item[(A1)] $\LP(X|B)-\LP(Y|B)\leq\sup\{X-Y|B\}$,
$\forall X|B, Y|B\in\dlin$.
\item[(A2)] $\LP(\lambda X|B)=\lambda\LP(X|B),
\forall X|B\in\dlin, \forall\lambda\geq 0$.
\item[(A3)] $\LP(X+Y|B)\geq\LP(X|B)+\LP(Y|B)$,
$\forall X|B$, $Y|B\in\dlin$.
\item[(A4)] $\LP(A(X-\LP(X|A\wedge B))|B)=0,
\forall X\in\xset,
\forall A, B\in\bsetp:
A\wedge B\neq\varnothing$.
\end{itemize}
\item[b)] $\lpr$ is convex on $\dlin$ if and only if (A1), (A4) and the following axiom hold  \cite[Theorem 8]{pel05}
\begin{itemize}
\item[(A5)] $\LP(\lambda X+(1-\lambda)Y|B)\geq\lambda\LP(X|B)+(1-\lambda)\LP(Y|B),
\forall X|B,Y|B\in\dlin, \forall \lambda\in ]0,1[$.\footnote{
Recall that the lower prevision $\lpr$ is termed convex referring to the convexity constraint $\sum_{i=1}^{m}s_i=1$ in Definition \ref{def:coh_conv_prev} b),
not to axiom (A5),
which actually tells us that $\lpr$ is concave, as a real functional.
}
\end{itemize}
\end{itemize}
\end{proposition}
\begin{remark}
\label{rem:2-convexity}
Exploiting some equivalences between axioms or groups of axioms,
Proposition \ref{pro:conditional convex prevision on linear space} as well as the later Propositions \ref{2-convexity_structure} and \ref{pro:conditional 2-convex prevision on linear space}
could be restated in a different form.
For instance, axiom (A1) is equivalent to the following
\begin{itemize}
\item[(A$\textit{1}'$)] If  $X|B,Y|B\in\dlin$, $\mu\in\rset$ are such that $X|B\geq Y|B+\mu$,
then $\LP(X|B)\geq\LP(Y|B)+\mu$.
\end{itemize}
Axiom (A1') is also equivalent to monotonicity plus translation invariance:
\begin{itemize}
\item[-] If  $X|B,Y|B\in\dlin$ and $X|B\geq Y|B$,
then $\LP(X|B)\geq\LP(Y|B)$ (monotonicity).
\item[-] If  $X|B\in\dlin$, $\mu\in\rset$,  
then $\LP(X+\mu|B)=\LP(X|B)+\mu$ (translation invariance).
\end{itemize}


Alternatively, (A1) may be replaced in Proposition \ref{pro:conditional convex prevision on linear space}
by $\lpr(X|B)\geq\inf(X|B)$, $\forall X|B\in\dlin$,
thus corresponding to the original version in \cite{wil07}.
\end{remark}
Condition (A4) is the \emph{Generalised Bayes Rule (GBR)},
introduced in \cite{wil75, wil07} and studied also in \cite{wal91} in the special case $B=\Omega$.

Since our discussion will focus on minimal consistency properties for a conditional lower prevision,
we have to mention a generalisation to a conditional framework of the implication (inclusion) relation between events,
which is termed Goodman--Nguyen relation ($\leq_{GN}$).
In fact, suppose $A\Rightarrow B$ (or $A\subseteq B$).
Then, asking that $\mu(A)\leq\mu(B)$ is a really minimal rationality requirement for any $\mu$ aiming at measuring how likely an event is,
given that, whenever event $A$ proves to be true,
$B$ comes true too.
The following extension of the implication to conditional events was proposed in \cite{goo88}:
\begin{eqnarray}
\label{eq:Goodman_Nguyen_events}
\begin{array}{lll}
A|B\leq_{GN} C|D &\mbox{ iff } A\wedge B\Rightarrow C\wedge D
&\mbox{ and } \nega{C}\wedge D\Rightarrow \nega{A}\wedge B.
\end{array}
\end{eqnarray}
The Goodman-Nguyen relation $\leq_{GN}$ was further extended to conditional gambles in \cite{pel14}:
\begin{eqnarray*}
\label{eq:Goodman_Nguyen_random_numbers}
\begin{array}{cc}
X|B\leq_{GN} Y|D \mbox{ iff }
I_{B}X+I_{\nega{B}\vee D}\sup(X|B)\leq I_{D}Y+I_{B\vee \nega{D}}\inf(Y|D)
\end{array}
\end{eqnarray*}
showing that $X|B\leq_{GN} Y|D$ implies $\lpr(X|B)\leq\lpr(Y|D)$ for a C-convex or coherent $\lpr$ \cite[Proposition 10]{pel14}.

\section{$2$-convex lower previsions}
\label{sec:2-convex_previsions}
In Definition \ref{def:coh_conv_prev}, a) and b),
there is no upper bound to $m\in\natset$.
One may think of introducing it as a natural way of weakening coherence and convexity.
More precisely,
let us call \emph{elementary gain} on $X_i|B_i$ any term $s_{i}B_{i}(X_{i}-\lpr(X_{i}|B_{i}))$,
with the proviso that $-B_0(X_0-\lpr(X_0|B_0))$ in
Definition \ref{def:coh_conv_prev} b) is also an elementary gain,
formally corresponding to $s_0=-1$.
Then,
we may decide that no more than $n$ elementary gains are allowed in either
$\LG$ (Definition \ref{def:coh_conv_prev}, a)) or $\LGC$ (Definition \ref{def:coh_conv_prev}, b)).
When doing so,
we speak of \emph{$n$-coherent} or \emph{$n$-convex} lower previsions.
This approach extends the notion of $n$-coherent (unconditional) prevision in \cite[Appendix B]{wal91}.

Intuition suggests that the smaller $n$ is,
the looser the corresponding consistency concept is.
In the extreme cases $n$ may be as small as $1$ with coherence,
$2$ with convexity.

However,
\emph{$1$-coherence} is too weak.
In fact, $\lpr$ is $1$-coherent on $\dset$ iff, $\forall X_0|B_0\in\dset$, $\forall s_0\in\rset$,
$\sup\{s_{0}B_{0}(X_{0}-\lpr(X_{0}|B_{0}))|B_{0}\}\geq 0$.

It is easy to see that this is equivalent to \emph{internality},
i.e. to requiring that $\lpr(X_0|B_0)\in [\inf(X_0|B_0),\sup(X_0|B_0)]$, $\forall X_0|B_0\in\dset$.
\begin{remark}
\emph{($1$-Avoiding Uniform Loss ($1$-AUL))}

A still weaker concept is that of $1$-Avoiding Uniform Loss ($1$-AUL).
Say that $\LP$ is $1$-AUL on $\dset$ iff $\forall X_0|B_0\in\dset$, $\forall s_0>0$,
\begin{eqnarray}
\label{eq:1_AUL}
\sup\{s_0 B_0(X_0-\LP(X_0|B_0))|B_0\}\geq 0.
\end{eqnarray}
The condition of $1$-AUL is equivalent to $\LP(X_0|B_0)\leq\sup(X_0|B_0)$, $\forall X_0|B_0\in\dset$.
In particular, this implies $\LP(0|B)\leq 0$, $\forall 0|B\in\dset$.

The wording $1$-AUL suggests its derivation from a concept of Avoiding Uniform Loss (AUL),
which may in fact be obtained from Definition \ref{def:coh_conv_prev} a) by replacing
`$m\in\natsetp$' with `$m\in\natset^{+}$'
(and consequently the gain $\LG$ with
$\LG_{AUL}=\sum_{i=1}^{m}s_i B_i(X_i-\lpr(X_i|B_i))$, $s_i\geq 0$).

The notion of $1$-AUL has an ancillary role in the theory of $2$-convex and $2$-coherent lower previsions:
it is a rather mild prerequisite to certain properties.
In this sense,
there is a similarity with the role of the condition of AUL for convex previsions
\cite{pel03,pel05} (cf. also Remark \ref{rem:finiteness_2_ne} in Section \ref{sec:2_coherent_natural_extension}).
\end{remark}
Since internality alone does not seem enough as a rationality requirement,
we turn our attention in this section to what seems to be the next weakest consistency notion,
that is $2$-convexity.\footnote{
\label{foo:1_become_2}
$2$-convex previsions were termed $1$-convex in \cite{bar09,pel14}.
Here we prefer the locution `$2$-convex' by analogy with the rule for fixing $n$
(as the number of elementary gains) in `$n$-coherent' in \cite{wal91}.
}
\begin{definition}
\label{def:2-convexity}
$\LP:\dset\rightarrow\rset$ is a \emph{2-convex} conditional lower prevision on $\dset$ iff,
$\forall X_0|B_0, X_1|B_1\in\dset$,
we have that,
defining $\LGDC=B_{1}(X_{1}-\LP(X_{1}|B_{1}))-B_{0}(X_{0}-\LP(X_{0}|B_{0}))$,
\begin{eqnarray}
\label{eq:cond_2-convexity}
\sup(\LGDC|B_{0}\vee B_{1})\geq 0.
\end{eqnarray}
\end{definition}

\subsection{Basic properties of $2$-convex lower previsions}
\label{sec:basic_2_convex}
We explore now some basic features of $2$-convex previsions.
Some critical aspects are discussed next,
showing in Section \ref{sec:centered_2-convex} that they can be solved resorting to the subclass of centered $2$-convex previsions.

A remarkable result in our framework is the characterisation of $2$-convexity
on a structured set $\dlin$.
\begin{proposition}
\label{2-convexity_structure}
A conditional lower prevision $\lpr:\dlin\rightarrow\rset$ is $2$-convex
on $\dlin$ if and only if (A1) and (A4) hold.
\begin{proof}
Suppose first that (A1) and (A4) hold.
Then, $\forall X_0|B_0, X_1|B_1\in\dlin$,
we obtain,
using (A1) at the first inequality and (A4) at the second,
\begin{eqnarray*}
\label{eq:proof_A1A4}
\begin{array}{ll}
&\sup\{B_1(X_1-\lpr(X_1|B_1))-B_0(X_0-\lpr(X_0|B_0))|B_0\vee B_1\}\geq \\ &\lpr(B_1(X_1-\lpr(X_1|B_1))|B_0\vee B_1)-\lpr(B_0(X_0-\lpr(X_0|B_0))|B_0\vee B_1)=0.
\end{array}
\end{eqnarray*}
Therefore, $\lpr$ is $2$-convex.

Conversely,
let $\lpr$ be $2$-convex.
Then, the proof that (A1) and (A4) hold is part of the proof of Theorem 8 in \cite{pel05}.
\end{proof}
\end{proposition}

To point out an important consequence of Proposition \ref{2-convexity_structure},
compare it with Proposition \ref{pro:conditional convex prevision on linear space} b).
It follows at once that the difference between $2$-convexity and convexity,
on $\dlin$,
is due to axiom (A5).
On the other hand,
the proof that a convex prevision on $\dlin$ must satisfy (A5),
given in \cite[Theorem 8]{pel05},
only involves a gain $\LGC$ made up of $3$ elementary gains,
i.e. it does not fully exploit convexity,
but only $3$-convexity.
This justifies the following conclusion:
\begin{proposition}
\label{pro:n_equivalence}
On $\dlin$, $n$-convexity with $n\geq 3$ and convexity are equivalent concepts.
\end{proposition}
Hence,
the very difference between convexity and $n$-convexity reduces to that between
convexity and $2$-convexity,
at least on $\dlin$.
Yet,
if $\lpr$ is defined on a set $\dset$ other than $\dlin$,
we may think of extending it to some $\dlin\supset\dset$.
If $\lpr$ is $n$-convex on $\dset$, $n\geq 3$,
and has an $n$-convex extension to $\dlin$,
then $\lpr$ is convex on $\dlin$ and therefore also on $\dset$.
It ensues that if $\lpr$ is $n$-convex ($n\geq 3$) but not convex on $\dset$,
$\lpr$ will have no $n$-convex extension on any sufficiently large superset of $\dset$
(any $\dset^{*}$ including some $\dlin$ containing $\dset$)
 - see also the later Example \ref{ex:non-extension} in Section \ref{sec:2_n_coherence}.
This is a negative aspect of $n$-convexity,
when $n\geq 3$.
More generally, the discussion above shows that $n$-convex previsions are not
particularly significant as an autonomous concept,
when $n\geq 3$.

\subsection{The $2$-convex natural extension}
\label{2_convex_natural_extension}
Turning again to $2$-convex previsions,
let us define a special extension,
the $2$-convex natural extension.
\begin{definition}
\label{def:2-convex_natext}
Given a lower prevision $\lpr:\dset\rightarrow\rset$
and an arbitrary conditional gamble $Z|B$,
let
\begin{eqnarray}
\label{eq:2-convex_setL}
\begin{array}{lll}
L(Z|B)=&\{\alpha:\sup\{A(X-\lpr(X|A))
-B(Z-\alpha)|A\vee B\}<0,\\
&\mbox{ for some } X|A\in\dset\}.
\end{array}
\end{eqnarray}
Then the \emph{$2$-convex natural extension} $\LEDC$ of $\lpr$ on $Z|B$ is
\begin{eqnarray}
\label{eq:2-convex_natext}
\LEDC(Z|B)=\sup L(Z|B).
\end{eqnarray}
\end{definition}
In general,
$L(Z|B)$ may be empty,
in which case $\LEDC(Z|B)=-\infty$,
following the usual convention for suprema.
When $L(Z|B)\neq\emptyset$,
it is instead possible that $\LEDC(Z|B)=+\infty$.
The results in the next proposition are helpful in hedging these two occurrences.
\begin{proposition}
\label{pro:2-convex_natext_sufficient}
\begin{itemize}
\item[a)] $L(Z|B)\neq\emptyset$, if $\exists\ Y|C\in\dset$ such that $C\Rightarrow B$.
\item[b)] Let $\lpr$ be $2$-convex and
such that $0|B\in\dset$ and $\lpr(0|B)=0$ $\forall X|B\in\dset$.
Given $0|C\notin\dset$,
the extension of $\lpr$ on $\dset\cup\{0|C\}$ such that $\lpr(0|C)=0$ is $2$-convex.
\item[c)] When $L(Z|B)\neq\emptyset$, $L(Z|B)=\ ]-\infty,\LEDC(Z|B)[$.
\item[d)] If $L(Z|B)\neq\emptyset$
and $\sup(X|A)\geq\lpr(X|A)$, $\forall X|A\in\dset$,
then $\LEDC(Z|B)\leq\sup(Z|B)$, $\forall Z|B$.
\item[e)] Let $\lpr$ be $2$-convex and $0|B\in\dset$, $\forall X|B\in\dset$.
Then, $\forall X|B\in\dset$, $\sup(X|B)\\ \geq\lpr(X|B)$
iff $\lpr(0|B)\leq 0$.
\end{itemize}
\end{proposition}
\begin{proof}
\emph{Proof of a).}
Identical to the proof of Proposition 6 in \cite{pel05}.

\emph{Proof of b).}
To check that the extension on $\dset\cup\{0|C\}$ with $\lpr(0|C)=0$
is $2$-convex,
we only have to check the suprema of two non-trivial gains in Definition \ref{def:2-convexity}:
the one arising from buying $X|B$ and selling $0|C$,
and that corresponding to buying $0|C$ and selling $X|B$.

In the former situation,
the gain is
\begin{eqnarray*}
\label{eq:gain_b}
\LGDC=B(X-\lpr(X|B))-C(0-\lpr(0|C))=B(X-\lpr(X|B)) - B(0-\lpr(0|B)),
\end{eqnarray*}
and
$\sup(\LGDC|B\vee C)\geq\sup(\LGDC|B)=\sup(B(X-\lpr(X|B))-B(0-\lpr(0|B))|B)\geq 0$,
using $2$-convexity of $\lpr$ on $\dset$ at the last inequality.

The latter situation can be treated analogously.

\emph{Proof of c).}
We show first that $L(Z|B)\subseteq\ ]-\infty,\LEDC(Z|B)[$.
Let $\alpha\in L(Z|B)$ such that
$s=\sup\{A(X-\lpr(X|A))-B(Z-\alpha)|A\vee B\}<0$ 
for some $X|A\in\dset$.
By \eqref{eq:2-convex_natext}, $\alpha\leq\LEDC(Z|B)$.
If \emph{ex absurdo} $\alpha=\LEDC(Z|B)$,
taking $\delta>0$ such that $s<s+\delta<0$,
we get 
$\sup\{A(X-\lpr(X|A))-B(Z-(\alpha+\delta))|A\vee B\}=
\sup\{A(X-\lpr(X|A))-B(Z-\alpha)+B\delta|A\vee B\}\leq
s+\sup\{B\delta|A\vee B\}\leq s+\delta<0$,
a contradiction.

Conversely,
let $\alpha\in ]-\infty,\LEDC(Z|B)[$.
Then there exists $\beta\in L(Z|B):\alpha<\beta\leq\LEDC(Z|B)$.
Further,
$\sup\{A(X-\lpr(X|A))-B(Z-\alpha)|A\vee B\}\leq\sup\{A(X-\lpr(X|A))-B(Z-\beta)|A\vee B\}<0$,
which implies §$\alpha\in L(Z|B)$.

\emph{Proof of d).}
We show that $\sup(Z|B)\notin L(Z|B)$.
Recalling \eqref{eq:2-convex_setL},
and since $-B(Z-\sup(Z|B))\geq 0$,
$\sup\{A(X-\lpr(X|A))-B(Z-\sup(Z|B))|A\vee B\}\geq
\sup\{A(X-\lpr(X|A))|A\vee B\}\geq\sup\{A(X-\lpr(X|A))|A\}=\sup(X|A)-\lpr(X|A)\geq 0$.
This means that $\sup(Z|B)\notin L(Z|B)$.
By c), $\sup(Z|B)\geq\LEDC(Z|B)$.

\emph{Proof of e).}
If $\sup(X|B)\geq\lpr(X|B)$,
$\forall X|B\in\dset$,
then in particular $\lpr(0|B)\leq\sup(0|B)=0$.

As for the reverse implication,
let $\lpr(0|B)\leq 0$.
Since $\lpr$ is $2$-convex, it holds that
$0\leq\sup\{B(X-\lpr(X|B))-B(0-\lpr(0|B))|B\}=\sup\{B(X-\lpr(X|B))+B\lpr(0|B)|B\}\leq
\sup\{X|B-\lpr(X|B)\}=\sup(X|B)-\lpr(X|B)$,
that is $\lpr(X|B)\leq\sup(X|B)$.
\end{proof}
Parts a) and b) of Proposition \ref{pro:2-convex_natext_sufficient} suggest a simple way
to ensure $\LEDC(Z|B)\neq -\infty$:
just add the gamble $0|B$ to $\dset$,
putting $\lpr(0|B)=0$.
To guarantee $\LEDC(Z|B)\neq +\infty$,
it is sufficient that any $0|C$ in $\dset$ (or added to $\dset$)
is given a non-positive lower prevision,
by d) and e).
Clearly, the simplest and most obvious choice is to put $\lpr(0|C)=0$, $\forall 0|C$.
This would make $\lpr$ a centered $2$-convex lower prevision;
in the remainder of this section we do not however rule out the possibility that $\lpr(0|C)\neq 0$ for some $0|C$.

The properties of the $2$-convex natural extension are very similar to those of the natural extension:
\begin{proposition}
\label{pro:natext_properties}
Let $\lpr:\dset\rightarrow\rset$ be a lower prevision,
with $\dset\subseteq\dlin$.
If $\LEDC$ is finite on $\dlin$, then
\begin{itemize}
\item[a)] $\LEDC(X|B)\geq\lpr(X|B)$, $\forall X|B\in\dset$.
\item[b)] $\LEDC$ is $2$-convex on $\dlin$.
\item[c)] If $\lpr^{*}$ is $2$-convex on $\dlin$ and $\lpr^{*}(X|B)\geq\lpr(X|B)$, $\forall X|B\in\dset$,
then $\lpr^{*}(X|B)\geq\LEDC(X|B)$, $\forall X|B\in\dlin$.
\item[d)] $\lpr$ is $2$-convex on $\dset$ if and only if $\LEDC=\lpr$ on $\dset$.
\item[e)] If $\lpr$ is $2$-convex on $\dset$, $\LEDC$ is its smallest $2$-convex extension on $\dlin$.
\end{itemize}
\end{proposition}
\begin{proof}
Assumptions a)\textdiv e) can be proven in the same way as the corresponding
a)\textdiv e) of Theorem 9 in \cite{pel05}
(regarding properties of the convex natural extension),
with some obvious notation changes and simplifications.
In the proof of a) replace $\LEC$ in \cite{pel05} with $\LEDC$.
The proof of b) checks that $\LEDC$ satisfies axioms (A1) and (A4)
(called (D1), (D3) in \cite{pel05}),
according to Proposition \ref{2-convexity_structure}.
This is done using some special gains,
simplifying those in \cite{pel05}:
the summation of the terms with stakes $s_1,\ldots,s_m$ is replaced by a single
term with stake $s_1=1$,
in agreement with Definition \ref{def:2-convexity}.
The proofs of c), d) are analogous,
while e) follows from c) and d).
\end{proof}
In words, the $2$-convex natural extension
dominates $\lpr$ (by a)),
characterises $2$-convexity (by d))
and is the least-committal
$2$-convex extension of $\lpr$ (by b), c), e)).

\subsection{Drawbacks of $2$-convexity}
\label{sec:drawbacks_2_convexity}
Being rather weak a consistency concept,
$2$-convexity may not satisfy a number of properties which necessarily hold for coherent lower previsions.
For instance,
the \emph{positive homogeneity} axiom (A2) of Proposition~\ref{pro:conditional convex prevision on linear space},
$\lpr(\lambda X|B)=\lambda\lpr(X|B)$, with $\lambda\geq 0$, may not hold,
not even weakening it to
\begin{eqnarray}
\label{eq:weak_pos_hom}
\lpr(\lambda X|B)\geq\lambda\lpr(X|B), \forall\lambda\in [0,1].
\end{eqnarray}
(Unconditional versions of \eqref{eq:weak_pos_hom} hold for centered convex previsions.)

It can instead be shown that
\begin{proposition}
\label{pro:2-convex_internality}
If, given $\lambda\in\rset$, $\lpr$ is $2$-convex on $\dset\supseteq\{X|B, \lambda X|B\}$,
then necessarily
\begin{eqnarray}
\label{eq:2-convex_internality}
\inf\{(\lambda-1)X|B\}+\lpr(X|B)\leq\lpr(\lambda X|B)
\leq\sup\{(\lambda-1)X|B\}+\lpr(X|B).
\end{eqnarray}
\end{proposition}
\begin{proof}
To obtain the first inequality,
apply Definition \ref{def:2-convexity} with $X_1|B_1=X|B$ and $X_0|B_0=\lambda X|B$:
\begin{eqnarray*}
\label{eq:proof_bounds_2_convex_1}
\begin{array}{ll}
\sup\{B(X-\lpr(X|B))-B(\lambda X-\lpr(\lambda X|B))|B\}\geq 0 \mbox{ iff}\\
\sup\{(1-\lambda)X|B\}-\lpr(X|B)+\lpr(\lambda X|B)\geq 0 \mbox{ iff}\\
\lpr(\lambda X|B)\geq\inf\{(\lambda-1)X|B)\}+\lpr(X|B).
\end{array}
\end{eqnarray*}
The proof of the second inequality is analogous
(let $X_1|B_1=\lambda X|B$, $X_0|B_0=X|B$ in Definition \ref{def:2-convexity}).
\end{proof}
Condition \eqref{eq:2-convex_internality} seems to be rather mild,
as the next example points out.

\begin{example}
\label{ex:internality}
Given $\dset=\{X|B, 2X|B\}$ $(\lambda=2)$,
where the image of $X|B$ is $[-1,1]$ and $\lpr(X|B)=0.2$,
equation \eqref{eq:2-convex_internality} gives the bounds
$\lpr(2X|B)\in [-0.8,1.2]$.
It is easy to check that $\lpr$ is $2$-convex on $\dset$
whatever is the choice for $\lpr(2X|B)$ in the interval $[-0.8,1.2]$.
Depending on the value for $\lpr(2X|B)$ selected in this interval,
it may be $\lpr(2X|B)\gtreqless 2\lpr(X|B)$.
\end{example}
An annoying feature of $2$-convexity is that \emph{internality} may fail,
i.e. $\lpr(X|B)$ need not belong to the closed interval $[\inf(X|B),\sup(X|B)]$.
Thus, $2$-convex previsions may not satisfy a property holding even for $1$-coherent previsions.

It has to be noticed that $2$-convexity permits no complete freedom in departing from internality.
There are two issues to be emphasized with respect to this question.
The first tells us that lack of internality cannot be two-sided,
because of the following result.
\begin{proposition}
\label{pro:2-convex_inf_sup_relation}
If $\lpr:\dset\rightarrow\rset$ is $2$-convex on $\dset$ and $\lpr(Y|D)<\inf(Y|D)$ for some $Y|D\in\dset$,
then $\lpr(X|B)\leq\sup(X|B)$, $\forall X|B\in\dset$.
Similarly,  $\lpr(Y|D)>\sup(Y|D)$ for some $Y|D\in\dset$ implies
$\lpr(X|B)\geq\inf( X|B)$, $\forall X|B\in\dset$.
\end{proposition}
\begin{proof}
We equivalently prove that there are no $X|B,Y|D\in\dset$
such that $\lpr(X|B)\\>\sup(X|B)$ and $\lpr(Y|D)<\inf(Y|D)$.

By contradiction, take $\delta,\epsilon>0$ and suppose
\begin{eqnarray}
\label{eq:proof_2_convex_bounds}
\lpr(X|B)=\sup(X|B)+\delta, \lpr(Y|D)=\inf(Y|D)-\epsilon.
\end{eqnarray}
Then,
$\LGDC|B\vee D=B(X-(\sup(X|B)+\delta))-D(Y-(\inf(Y|D)-\epsilon))|B\vee D$
is such that $\sup(\LGDC|B\vee D)<0$.
In fact, $\sup(\LGDC|B\vee D)=\max\{\sup(\LGDC|\nega{B}\wedge D),\sup(\LGDC|B\wedge\nega{D}),\sup(\LGDC|B\wedge D)\}$
and we have:
\begin{itemize}
\item $\sup(\LGDC|\nega{B}\wedge D)=\sup(-D(Y-\inf(Y|D)+\epsilon)|\nega{B}\wedge D)=
\sup(-Y|\nega{B}\wedge D)+\inf(Y|D)-\epsilon=\inf(Y|D)-\inf(Y|\nega{B}\wedge D)-\epsilon\leq -\epsilon<0$;
\item $\sup(\LGDC|B\wedge\nega{D})=\sup(X|B\wedge\nega{D})-\sup(X|B)-\delta\leq -\delta<0$;
\item $\sup(\LGDC|B\wedge D)=\sup(X|B\wedge D)-\sup(X|B)-\delta+\inf(Y|D)-\inf(Y|B\wedge D)-\epsilon\leq -\delta-\epsilon<0$.
\end{itemize}
Therefore,
any $\lpr$ satisfying \eqref{eq:proof_2_convex_bounds} is not $2$-convex,
according to Definition~\ref{def:2-convexity}.
\end{proof}
The second issue is that $2$-convexity imposes a sort of, so to say, two-component internality.
To see this,
note that
\begin{lemma}
\label{lem:2-convex_inf_sup_bounds}
If $\lpr:\dset\rightarrow\rset$ is $2$-convex on $\dset$,
and $X|B$, $Y|B\in\dset$,
then
\begin{eqnarray}
\label{eq:2-convex_inf_sup_bounds}
\begin{array}{ll}
\inf\{X-Y|B\}\leq\lpr(X|B)-\lpr(Y|B)
\leq\sup\{X-Y|B\}.
\end{array}
\end{eqnarray}
\end{lemma}
\begin{proof}
The second inequality in \eqref{eq:2-convex_inf_sup_bounds} is axiom (A1),
a necessary condition for $2$-convexity
which implies also the first inequality.
In fact,
$\inf\{X-Y|B\}=-\sup\{Y-X|B\}\leq -(\lpr(Y|B)-\lpr(X|B))=\lpr(X|B)-\lpr(Y|B)$.
\end{proof}
Recall now that $\lpr(X|B)$ is interpreted as a supremum buying price for $X|B$,
and that Definition \ref{def:2-convexity} ensures that buying $X|B$ for $\lpr(X|B)$
and selling $Y|B$ at its supremum buying price $\lpr(Y|B)$ would be (marginally)
acceptable for $2$-convexity.
Then,
equation \eqref{eq:2-convex_inf_sup_bounds} tells us that the profit $\lpr(X|B)-\lpr(Y|B)$
from this two-component exchange ($X|B$ vs. $Y|B$) guarantees no arbitrage.
For instance,
it cannot exceed the income upper bound $\sup\{X-Y|B\}$.

As a further questionable feature of $2$-convexity,
the Goodman-Nguyen relation may not induce an agreeing
ordering on a $2$-convex prevision.
This is tantamount to saying that the partial ordering of some $2$-convex conditional
previsions may conflict with the ordering of the extended implication (inclusion) relation $\leq_{GN}$.

For instance,
from \eqref{eq:Goodman_Nguyen_events},
if $B\Rightarrow C$ then $0|C\leq_{GN} 0|B$.
Agreement with the Goodman-Nguyen relation requires
$\lpr(0|C)\leq\lpr(0|B)$ to hold,
but it can be proven that if $\lpr(0|B)<0$ and $B\Rightarrow C$,
then $2$-convexity asks instead that $\lpr(0|C)\geq\lpr(0|B)$
(the inequality may be strict).

\subsection{Centered $2$-convex lower previsions}
\label{sec:centered_2-convex}
The critical issues of $2$-convexity discussed in the preceding section
can be solved or softened requiring the additional property
\begin{eqnarray*}
\label{eq:centering_condition}
\forall X|B\in\dset, 0|B\in\dset \mbox{ and } \lpr(0|B)=0,
\end{eqnarray*}
i.e. restricting our attention to centered $2$-convex
conditional lower previsions.
This is shown in the following proposition.
\begin{proposition}
\label{pro:C-convex_properties}
Let $\lpr:\dset\rightarrow\rset$ be a centered $2$-convex lower prevision on $\dset$.
Then,
\begin{itemize}
\item[a)] $\forall X|B\in\dset$, $\lpr(X|B)\in [\inf (X|B), \sup (X|B)]$.
\item[b)] $\lpr$ has a finite $2$-convex natural extension $\LEDC$ on any superset of $\dset$.
\item[c)] $X|B\leq_{GN} Y|D$ implies $\lpr(X|B)\leq\lpr(Y|D)$.
\end{itemize}
\end{proposition}
\begin{proof}
\emph{Proof of a).}
Put $Y|B=0|B$ and $\lpr(0|B)=0$ in \eqref{eq:2-convex_inf_sup_bounds}.

\emph{Proof of b).}
The statement follows from Proposition \ref{pro:2-convex_natext_sufficient}.

\emph{Proof of c).}
Proven in \cite[Proposition 10]{pel14} for C-convex previsions.
As noted in the Discussion following Proposition 10 in \cite{pel14},
the very same proof applies to centered $2$-convex previsions too
(cf. also Footnote \ref{foo:1_become_2}).
\end{proof}
\textit{Comment.}
The condition $\lpr(0|B)=0$ seems to be obvious,
and in fact guarantees more satisfactory properties to $2$-convexity.
In our view,
the main reason for considering the alternative $\lpr(0|B)\neq 0$ is to encompass additional uncertainty models.
This is patent already in the unconditional framework:
convex risk measures, as introduced in \cite{fol02,fol02bis},
correspond to convex, not necessarily centered previsions
\cite{pel03}.

Note that centered $2$-convexity implies $1$-coherence,
by Proposition \ref{pro:C-convex_properties} a),
while being obviously implied by $2$-coherence.
Hence,
the centering condition $\lpr(0|B)=0$ may be regarded as a technical instrument to guarantee that
the lower prevision $\lpr$ ensures more satisfactory properties than a generic $2$-convex prevision,
without having to assume $2$-coherence.

\section{$2$-coherent lower previsions}
\label{sec:2_coherent}
Our next step is a discussion of which additional properties are achieved by a $2$-coherent lower prevision.
\begin{definition}
\label{def:cond_2-coherence}
$\lpr:\dset\rightarrow\rset$ is a $2$-coherent lower prevision on $\dset$ iff
$\forall X_0|B_0$, $X_1|B_1\in\dset$, $\forall s_1\geq 0$, $\forall s_0\in\rset$,
defining $\Ss=\bigvee\{B_{i}:s_{i}\neq 0, i=0,1\}$,
$\LGD=s_{1}B_{1}(X_{1}-\LP(X_{1}|B_{1}))-s_{0}B_{0}(X_{0}-\LP(X_{0}|B_{0}))$
we have that, whenever $\Ss\neq\varnothing$,
\begin{eqnarray}
\label{eq:cond_2-coherence}
\sup\{\LGD|\Ss\}\geq 0.
\end{eqnarray}
\end{definition}
$2$-coherent lower previsions are characterized on $\dlin$ as follows:
\begin{proposition}
\label{pro:conditional 2-convex prevision on linear space}
Let $\LP:\dlin\rightarrow\rset$ be a conditional lower prevision.
$\lpr$ is $2$-coherent on $\dlin$ if and only if
(A1), (A2), (A4) and the following axiom hold:
\begin{itemize}
\item[(A6)] $\lpr(\lambda X|B)\leq \lambda\lpr(X|B)$, $\forall\lambda<0$.
\end{itemize}
\end{proposition}
\begin{proof}
We prove first that if (A1), (A2), (A4) and (A6) hold,
then $\lpr$ is $2$-coherent on $\dlin$.
Recalling for this Definition \ref{def:cond_2-coherence},
take any two $X_0|B_0, X_1|B_1\in\dlin$,
and any $s_1\geq 0$, $s_0\in\rset$.
Then,
using (A1) at the first inequality,
(A2) (when $s_0\geq 0$) or (A2) and (A6) (when $s_0<0$) at the second inequality,
we obtain:
\begin{eqnarray*}
\label{eq:cond_2-coherence_proof}
\begin{array}{lll}
\sup\{[s_1 B_1(X_1-\lpr(X_1|B_1))]-[s_0 B_0(X_0-\lpr(X_0|B_0))]|\Ss\}\geq \\
\lpr(s_1 B_1(X_1-\lpr(X_1|B_1))|\Ss)-\lpr(s_0 B_0(X_0-\lpr(X_0|B_0))|\Ss)\geq \\
s_1 \lpr(B_1(X_1-\lpr(X_1|B_1))|\Ss)-s_0 \lpr(B_0(X_0-\lpr(X_0|B_0))|\Ss)=0,
\end{array}
\end{eqnarray*}
where the equality holds because,
when $s_i\neq 0$,
$s_i \lpr(B_i (X_i-\lpr(X_i|B_i))|\Ss)=s_i \lpr(B_i (X_i-\lpr(X_i|B_i \wedge\Ss))|\Ss)=0\ (i=1,2)$ by (A4).

Conversely, if $\lpr$ is $2$-coherent, therefore also $2$-convex, on $\dlin$,
(A1) and (A4) hold by Proposition \ref{2-convexity_structure}.
Hence, it only remains to prove (A2) and (A6).

We prove first (A6).
Apply Definition \ref{def:cond_2-coherence},
with $X_1|B_1=X|B$, $X_0|B_0=\lambda X|B$,
$s_1=1, s_0=\frac{1}{\lambda}<0$:
$\sup(B(X-\lpr(X|B))-\frac{1}{\lambda}B(\lambda X-\lpr(\lambda X|B))|B)=
\sup(-\lpr(X|B)\\+\frac{1}{\lambda}\lpr(\lambda X|B))\geq 0$,
which is equivalent to $\lpr(\lambda X|B)\leq\lambda\lpr(X|B)$.

As for (A2),
consider the same assumptions of the proof of (A6).
Since now $s_0=\frac{1}{\lambda}>0$,
we obtain the inequality $\lpr(\lambda X|B)\geq\lambda\lpr(X|B)$.
Assuming instead $X_1|B_1=\lambda X|B$, $X_0|B_0=X|B$, $s_1=\frac{1}{\lambda}$,
$s_0=1$,
we obtain the reverse inequality $\lpr(\lambda X|B)\leq\lambda\lpr(X|B)$.
\end{proof}
\textit{Comment} A comparison of Propositions \ref{2-convexity_structure} and \ref{pro:conditional 2-convex
prevision on linear space} is useful for detecting two major
differences between (centered) $2$-convex and $2$-coherent
previsions.

One is positive homogeneity (axiom (A2)),
a condition which, on any set $\dset$, is necessary for $2$-coherence, but not for $2$-convexity.
The need for positive homogeneity depends on the specific model we wish to consider.
We might be willing to reject it in some instance,
typically because of \emph{liquidity risk} considerations.
Basically, this means that for a large positive $\lambda$ difficulties might be encountered at exchanging
$\lambda X|B$ at a price $\lpr(\lambda X|B)=\lambda\lpr(X|B)$,
because of lack of market liquidity at some degree.

The second difference is pointed out by axiom (A6).
To fix its meaning,
recall that given $\lpr(X|B)$,
its \emph{conjugate} upper prevision $\upr(X|B)$ is defined by
\begin{eqnarray}
\label{eq:conjugate}
\upr{(X|B)}=-\lpr(-X|B).
\end{eqnarray}
Hence, 
axiom (A6) ensures by \eqref{eq:conjugate} that
\begin{eqnarray*}
\label{eq:dominance}
\upr(X|B)\geq\lpr(X|B), \forall X|B\in\dlin.
\end{eqnarray*}

Therefore, $2$-coherence is preferable to $2$-convexity whenever we fix an upper ($\upr$) and a lower ($\lpr$) bound for the
uncertainty evaluation of $X|B$,
while keeping positive homogeneity.


\subsection{$2$-coherence versus $n$-coherence}
\label{sec:2_n_coherence}
Compare Propositions \ref{pro:conditional 2-convex prevision on linear space} and 1, a).
Recalling that (A6) is a necessary condition for $2$-coherence and hence also for coherence,
only the superlinearity axiom (A3) distinguishes $2$-coherence and coherence on $\dlin$.
From this,
deductions on the role of $n$-coherence, $n\geq 3$,
can be made which are quite analogue to those on $n$-convexity in Section \ref{sec:2-convex_previsions}.
This time,
it can be shown that any $n$-coherent lower prevision, $n\geq 3$,
must satisfy (A3),
and hence that:
\begin{proposition}
\label{pro:2_n_coherence}
On $\dlin$, $n$-coherence with $n\geq 3$ and coherence are equivalent concepts.
\end{proposition}
And again,
we may in general argue that $n$-coherence has no special relevance, compared to coherence, when $n\geq 3$.
In particular,
$n$-coherent extensions of an $n$-coherent $\lpr$ exist on sufficiently large sets if and only if $\lpr$ is coherent.

The latter concept is illustrated in the next example,
elaborating on Example 2.7.6 in \cite{wal91}.
\begin{example}
\label{ex:non-extension}
Let $\prt=\{a,b,c,d\}$ be a partition of the sure event $\Omega$.
Define $\lpr$ on the powerset of $\prt$ as follows:
\begin{itemize}
\item $\lpr(\Omega)=1$
\item $\lpr(E)=\frac{1}{2}$ if $E$ is made up of $2$ or $3$ elements of $\prt$,
one of which is $a$.
\item $\lpr(E)=0$ otherwise.
\end{itemize}
It is shown in \cite{wal91} that $\lpr$ is not coherent,
while being $3$-coherent,
and hence also $3$-convex.
We show now that \emph{$\lpr$ has no $3$-convex extension to
the linear space $\mathcal{L}(\prt)$ of all gambles defined on $\prt$}.

In fact, suppose a $3$-convex extension, also termed $\lpr$, exists,
and define
$A=a$, $B=a\vee b$, $C=a\vee c$, $D=a\vee d$.
Note that, by applying (A1) with
$X=\frac{1}{2}A$, $Y=A$ and $B=\Omega$,
we get $\lpr(\frac{1}{2}A)\leq\lpr(A)+\sup(-\frac{1}{2}A)=\lpr(A)=0$.
Therefore, also the $3$-convex extension of $\lpr$ to
$\frac{1}{4}(B+C+D-1)=\frac{1}{2}A$ should be non-positive.
Note also that $\lpr(-1)=-1$
(use (A1) with $X=0$, $Y=-1$, $B=\Omega$, to
get $\lpr(-1)\geq -1$,
which is what is needed next;
interchanging $X$ and $Y$ in (A1) gives also $\lpr(-1)\leq -1$).
By applying axiom (A5) as a necessary condition for $3$-convexity,
we obtain
$\lpr(\frac{1}{4}(B+C+D-1))=
\lpr(\frac{1}{2}(\frac{1}{2}B+\frac{1}{2}C)+\frac{1}{2}(\frac{1}{2}D-\frac{1}{2}))
\geq\frac{1}{2}\lpr(\frac{1}{2}B+\frac{1}{2}C)+\frac{1}{2}\lpr(\frac{1}{2}D-\frac{1}{2})
\geq\frac{1}{4}\lpr(B)+\frac{1}{4}\lpr(C)+\frac{1}{4}\lpr(D)+\frac{1}{4}\lpr(-1)
\geq 3\cdot\frac{1}{4}\cdot\frac{1}{2}-\frac{1}{4}=\frac{1}{8}>0$,
a contradiction.

From what we have just proven,
we may conclude that:
\begin{itemize}
\item[a)] the given $\lpr$ on the powerset of $\prt$ has no $3$-convex extension to $\mathcal{L}(\prt)$;
\item[b)] $\lpr$ (viewed now as $3$-coherent on the powerset of $\prt$) has
no $3$-coherent extension on $\mathcal{L}(\prt)$ either:
if it had one,
this extension would be $3$-convex too, contradicting a).
\end{itemize}
\end{example}
It is interesting to realise that, on $\dlin$,
convexity is what is missing to $2$-coherence (and viceversa) to achieve coherence:
\begin{proposition}
$\lpr:\dlin\rightarrow\rset$
is coherent iff it is both $2$-coherent and convex. 
\end{proposition}
\begin{proof}
Clearly,
coherence implies both $2$-coherence and convexity.

Conversely,
let $\lpr$ be both $2$-coherent and convex on $\dlin$.
By Propositions \ref{pro:conditional convex prevision on linear space} and \ref{pro:conditional 2-convex prevision on linear space},
it only remains to check that (A3) holds to ensure coherence of $\lpr$.
In fact, by (A5) and (A2), we have
\begin{eqnarray}
\label{eq:convexity_plus_2_coherence}
\begin{array}{lll}
\lpr(X+Y|B)&=\lpr(\frac{1}{2}(2X)+\frac{1}{2}(2Y)|B)\geq\\
&\frac{1}{2}\lpr(2X|B)+\frac{1}{2}\lpr(2Y|B)=\lpr(X|B)+\lpr(Y|B).
\end{array}
\end{eqnarray}
\end{proof}

\subsection{The $2$-coherent natural extension}
\label{sec:2_coherent_natural_extension}

$2$-coherent lower probabilities,
being also centered $2$-convex and $1$-AUL,
always have a $2$-convex natural extension.

Appreciably,
they further ensure the existence of a $2$-coherent natural extension.
This tells us that the additional properties of $2$-coherence are \emph{stable},
in the sense that they can be preserved by extension to any set of conditional gambles.

The role of the $2$-coherent natural extension is analogous for $2$-coherence to that of the $2$-convex natural extension for $2$-convexity,
and most derivations are quite similar.
We shall demonstrate in detail only the most differing ones.
\begin{definition}
\label{def:L_2}
Given a lower prevision $\lpr:\dset\rightarrow\rset$ and an
arbitrary $Z|B$,
let
\begin{eqnarray*}
	\label{eq:L_2}
	\begin{array}{lll}
		\LD(Z|B)=&\{\alpha:\sup\{s_1 A(X-\lpr(X|A))-B(Z-\alpha)|\Ss\}<0\\
		&\mbox{ for some } X|A\in\dset, s_1\geq 0\}.
	\end{array}
\end{eqnarray*}
where
\begin{eqnarray*}
\label{eq:s_2}
\Ss=\left\{
\begin{array}{ll}
A\vee B &\mbox{ if } s_1>0\\
B &\mbox{ if } s_1=0.
\end{array}
\right.
\end{eqnarray*}
Then,
the $2$-coherent natural extension $\LED$ of $\lpr$ on $Z|B$ is
\begin{eqnarray*}
\label{eq:2_coherent_nat_ext}
\LED(Z|B)=\sup\LD(Z|B).
\end{eqnarray*}
\end{definition}
\begin{proposition}[\emph{Existence of the $2$-coherent natural extension}]
\label{pro:2_coherent_nat_ext_existence}
Given a lower prevision $\lpr:\dset\rightarrow\rset$,
\begin{itemize}
\item[a)]
$\LD(Z|B)$ is non-empty, $\forall Z|B$.
\item[b)]
$\LD(Z|B)=]-\infty,\LED(Z|B)[$.
\item[c)]
If $\lpr$ is $1$-AUL,
$\LED(Z|B)\leq\sup(Z|B)$, $\forall Z|B$.
\item[d)]
If $\lpr$ is not a $1$-AUL prevision,
$\exists Z|B\in\dset$ such that $\LED(Z|B)=+\infty$.
\end{itemize}
\end{proposition}
\begin{proof}
\emph{Proof of a).}
Take $\alpha<\inf(Z|B)$ and let $X|A$ be any conditional gamble in $\dset$.
Then, putting $s_1=0$ in Definition \ref{def:L_2},
$\sup\{s_1 A(X-\lpr(X|A))-B(Z-\alpha)|\Ss\}=
\sup\{-B(Z-\alpha)|B\}=\alpha+\sup\{-Z|B\}=\alpha-\inf\{Z|B\}<0$,
i.e. $\alpha\in\LD(Z|B)$.
%
%

\emph{Proofs of b) and c).}
Same as the proofs of,
respectively,
Proposition \ref{pro:2-convex_natext_sufficient}c) and Proposition \ref{pro:2-convex_natext_sufficient}d),
replacing $A(X-\lpr(X|A))$ with $s_1 A(X-\lpr(X|A))$.

\emph{Proof of d).}
Let $\lpr$ be not $1$-AUL.
Then,
$\exists\ Z|B\in\dset, s>0$
such that
$\sup\{sB(Z-\lpr(Z|B))|B\}<0$.
Clearly, this implies $\lpr(Z|B)-\sup\{Z|B\}>0$.

Then, $\forall s_1>1$, $\forall\alpha$ such that $\alpha<\sup\{Z|B\}+s_1(\lpr(Z|B)-\sup\{Z|B\})$,
we get $\sup\{s_1 B(Z-\lpr(Z|B))-B(Z-\alpha)|B\}=
\alpha-\sup\{Z|B\} - s_1(\lpr(Z|B)-\sup\{Z|B\})<0$,
i.e. $\alpha\in\LD(Z|B)$.
Since $\alpha$ can be chosen arbitrarily large in $\LD(Z|B)$
by increasing $s_1$,
$\LED(Z|B)=+\infty$.
\end{proof}
\begin{remark}
\label{rem:finiteness_2_ne}
Proposition \ref{pro:2_coherent_nat_ext_existence} ensures that the condition of $1$-AUL is necessary and sufficient for the finiteness of the $2$-coherent natural extension $\LED$.
\end{remark}
\begin{proposition}[\emph{Properties of the $2$-coherent natural extension}]
\label{pro:2_coherent_nat_ext_prop}
Let $\dset\subseteq\dlin$ and $\lpr:\dset\rightarrow\rset$.
If $\LED$ is finite on $\dlin$, then
\begin{itemize}
\item[a)]
$\LED(X|B)\geq\lpr(X|B), \forall X|B\in\dset$.
\item[b)]
$\LED$ is $2$-coherent on $\dlin$.
\item[c)]
If $\lpr^\prime$ is $2$-coherent on $\dlin$ and
$\lpr^\prime (X|B)\geq\lpr(X|B), \forall X|B\in\dset$,
then $\lpr^\prime\geq\LED$ on $\dlin$.
\item[d)]
$\lpr$ is $2$-coherent on $\dset$ if and only if $\LED=\lpr$ on $\dset$.
\item[e)]
If $\lpr$ is $2$-coherent on $\dset$,
$\LED$ is its smallest $2$-coherent extension on $\dlin$.
\end{itemize}
\end{proposition}
\begin{proof}
\emph{Proof of a).}
See the proof of Theorem 9, a) in \cite{pel05}.

\emph{Proof of b).}
By Proposition \ref{pro:conditional 2-convex prevision on linear space},
we prove that $\LED$ satisfies axioms (A1), (A2), (A4), (A6).

(A1) and (A4):
see the proof that, respectively,
axioms (D1) and (D3) hold in \cite{pel05},
Theorem 9, b)
(with the obvious modifications
$\LG_1 = s_1 B_1(X_1 - \lpr(X_1|B_1))$, $s_1\geq 0$).

(A2):
The proof corresponds to that for axiom (A2) in \cite{pel09},
Theorem 3, p. 625 (conditioning now $\LG_1$ on $\Ss$).

(A6):
\emph{To prove (A6)},
follow the next two steps:
\begin{itemize}
\item[1)]
It should be proven that
\begin{eqnarray*}
	\LED(\lambda X|B)\leq\lambda\LED(X|B), \forall\lambda<0.
\end{eqnarray*}
Since we already know that (A2) holds,
we can use it to write
$\LED(\lambda X|B)=\LED(-\lambda(-X)|B)=-\lambda\LED(-X|B)
\leq\lambda\LED(X|B)$
if and only if
$-\LED(-X|B)\\ \geq\LED(X|B)$.

Therefore,
we can equivalently prove instead that
\begin{eqnarray}
	\label{eq:sum_opposites}
	\LED(X|B)+\LED(-X|B)\leq 0.
\end{eqnarray}
\item[2)]
We prove \eqref{eq:sum_opposites}.
Take arbitrarily $\alpha^+\in\LD(X|B)$, $\alpha^-\in\LD(-X|B)$.
Letting $g_i=s_i B_i(X_i-\lpr(X_i|B_i))$, $s_i\geq 0$
for $i=1,2$,
we have
by the definition of $\LD(X|B)$, $\LD(-X|B)$:
\begin{eqnarray*}
	\begin{array}{ll}
	\sup(g_1-B(X-\alpha^+)|B_1\vee B)=\sup(Z_1|B_1\vee B)<0,\\
	\sup(g_2-B(-X-\alpha^-)|B_2\vee B)=\sup(Z_2|B_2\vee B)<0.
	\end{array}
\end{eqnarray*}
Defining $H=B\vee B_1\vee B_2$,
it holds also that
\begin{eqnarray}
\label{eq:sup_sum}
\sup(Z_1 + Z_2|H)=\sup(g_1+g_2+B(\alpha^+ + \alpha^-)|H)<0.
\end{eqnarray}
In fact,
decompose $H$ into the sum of $4$ disjoint events as follows:
\begin{eqnarray}
\label{eq:cond_decomposition}
H=B\vee(B_1\wedge B_2\wedge\nega{B})
\vee(\nega{B_1}\wedge B_2\wedge\nega{B})
\vee(B_1\wedge\nega{B_2}\wedge\nega{B}).
\end{eqnarray}
Now condition $Z_1$ and $Z_2$ on any of the $4$ events in \eqref{eq:cond_decomposition} that are not impossible.
Let $H_j$ be the generic such event and $J=\{j\in\{1,2,3,4\}: H_j\neq\varnothing\}$.
Then note that $\sup \{Z_i|H_j\}\leq 0, (i=1,2)$.
In fact,
considering $Z_1$,
if $H_j$ is any of $B$, $B_1\wedge B_2\wedge\nega{B}$ or
$B_1\wedge\nega{B_2}\wedge\nega{B}$,
then $\sup\{Z_1|H_j\}\leq\sup\{Z_1|B_1\vee B\}<0$,
whilst $\sup\{Z_1|\nega{B_1}\wedge B_2\wedge\nega{B}\}=
\sup\{0|\nega{B_1}\wedge B_2\wedge\nega{B}\}=0$.
Similarly,
$\sup\{Z_2|H_j\}\leq 0$,
with equality iff $H_j=B_1\wedge\nega{B_2}\wedge\nega{B}$.
It ensues also that the two suprema $\sup\{Z_1|H_j\}$, $\sup\{Z_2|H_j\}$  cannot be simultaneously null, for $j\in J$.

Hence,
$\sup(Z_1|H)+\sup(Z_2|H)=\max\{\sup(Z_1|H_j)+\sup(Z_2|H_j), j\in J\}<0$.
The inequality \eqref{eq:sup_sum} follows,
since $\sup(Z_1+Z_2|H)\leq\sup(Z_1|H)+\sup(Z_2|H)$.
Further,
\begin{eqnarray*}
\label{eq:sup_not_less}
\sup(g_1 + g_2|H)\geq\sup(g_1 + g_2|B_1\vee B_2)\geq 0.
\end{eqnarray*}
using $2$-coherence of $\lpr$ on $\dset$ at the last inequality.

Therefore,
in order for inequality \eqref{eq:sup_sum} to hold,
necessarily $\alpha^+ + \alpha^- < 0$,
i.e. $\alpha^+ < -\alpha^-$,
$\forall \alpha^+\in\LD(X|B), \alpha^-\in\LD(-X|B)$.
Equivalently,
\begin{eqnarray*}
	\begin{array}{ll}
		\sup\{\alpha^+\in\LD(X|B)\}=\LED(X|B)\leq
		\inf\{-\alpha^-:\alpha^-\in\LD(-X|B)\}=\\
		-\sup\{\alpha^-\in\LD(-X|B)\}=
		-\LED(X|B),
	\end{array}
\end{eqnarray*}
which gives \eqref{eq:sum_opposites}.
\end{itemize}

\emph{Proofs of c) and d).}
Analogous to the proof of Theorem 9, c) and d) in \cite{pel05}.

\emph{Proof of e).}
Implied by c) and d).
\end{proof}
We may thus conclude that centered $2$-convexity and $2$-coherence appear to be the most significant weakenings
of (centered) convexity and coherence.

\subsection{About the Generalised Bayes Rule}
\label{subsec:GBR}
By Propositions \ref{2-convexity_structure} and \ref{pro:conditional 2-convex prevision on linear space},
the Generalised Bayes Rule (GBR) is a necessary consistency
condition for both $2$-convex and $2$-coherent lower previsions.
This guarantees that this key updating rule holds even with weaker
consistency concepts than coherence or convexity. However, it would
be erroneous to believe that nothing about the GBR changes with such
looser consistency requirements. To see this, put $B = \Omega$ in
(A4), getting
\begin{eqnarray*}
\label{eq:simple_GBR}
\lpr(A(X-\lpr(X|A)))=0,
\end{eqnarray*}
which informs us that $\lpr(X|A)$ is a solution of the equation
\begin{eqnarray}
\label{eq:GBR_equation}
\lpr(A(X-r))=0.
\end{eqnarray}
From Proposition 9 in \cite{pel05},
we know that if $\lpr$ is convex on $\dset\supset\{A, X|A,A(X-\lpr(X|A))\}$ and $\lpr(A)>0$,
then $\lpr(X|A)$ is the \emph{unique} solution of \eqref{eq:GBR_equation}.
A uniqueness result for coherent lower previsions is given in \cite{wal91}, Sec. 6.4.1.

With $2$-coherent or $2$-convex lower previsions,
$\lpr(X|A)$ may no longer be the unique solution of \eqref{eq:GBR_equation}.
The next result illustrates this for $2$-coherence.
\begin{proposition}
\label{pro:GBR_not_unique} Let $\lpr:D=\{A, X|A, A(X-r),
A(X-q)\}\rightarrow\rset$ be a lower prevision, such that $r\neq q$,
$\lpr(A(X-r))=\lpr(A(X-q))=0$, $A\neq\Omega$, $1 \geq \lpr(A)>0$.
Then $\lpr$ is $2$-coherent on $\dset$ if and only if
\begin{eqnarray}
\label{eq:in_interval}
\lpr(X|A), r, q\in [\inf(X|A),\sup(X|A)].
\end{eqnarray}
\end{proposition}
\begin{proof}
Suppose first that \eqref{eq:in_interval} holds.
To prove that $\lpr$ is $2$-coherent on $\dset$,
we may check by Definition \ref{def:cond_2-coherence} that
any admissible gain $\LGD$ satisfies \eqref{eq:cond_2-coherence}.
For this,
we consider the gains from betting on all couples of elements of $\dset$
(and their special cases $s_0=0$ or $s_1=0$,
where the effective bet is on a single element).
These gains may be partitioned into two groups:
\begin{enumerate}
\item
Gains from bets on the couples $(X|A, A(X-r))$,
$(X|A, A(X-q))$, $(A(X-r), A(X-q))$.

The proofs that all such gains satisfy \eqref{eq:cond_2-coherence}
are very similar for all couples. To exemplify, take the couple
$(X|A, A(X-r))$. Any admissible gain is either of
\begin{eqnarray*}
\begin{array}{lll}
\LGD=s_1 A(X-\lpr(X|A))-s_0 A(X-r);\\
\LGD^\prime=s_1 A(X-r)-s_0 A(X-\lpr(X|A)).
\end{array}
\end{eqnarray*}
Let us first look at $\LGD$.
If $s_0\neq 0$, $\LGD|\Ss=\LGD|\Omega$,
and $\sup\LGD\geq\LGD(\nega{A})=0$.
If $s_0=0$ (and $s_1>0$),
$\sup(\LGD|\Ss)=\sup(\LGD|A)=s_1\sup(X|A-\lpr(X|A))\geq 0$ by
\eqref{eq:in_interval}.

Consider now $\LGD^\prime$. If $s_1\neq 0$,
$\LGD^\prime|\Ss=\LGD^\prime|\Omega$, and
$\sup\LGD^\prime\geq\LGD^\prime(\nega{A})=0$. Let then $s_1=0$,
hence $\Ss=A$.

If $s_0>0$, $\sup(\LGD^\prime|\Ss)=\sup(-s_0 A(X-\lpr(X|A))|A)=
s_0\sup(\lpr(X|A)-X|A)=s_0(\lpr(X|A)-\inf(X|A))\geq 0$ by
\eqref{eq:in_interval}.

If $s_0<0$, $\sup(\LGD^\prime|\Ss)=\sup(-s_0 A(X-\lpr(X|A))|A)= -s_0
(\sup(X|A)-\lpr(X|A))\geq 0$, again by \eqref{eq:in_interval}.
\item
Gains from betting on one of the remaining three couples.

All such couples include $A$,
and the proof is identical for each of them.
Take for instance the couple $(A, X|A)$.
The related gains are
\begin{eqnarray*}
\begin{array}{lll}
\LGD=s_1 (A-\lpr(A))-s_0 A(X-\lpr(X|A));\\
\LGD^\prime=s_1 A(X-\lpr(X|A))-s_0 (A-\lpr(A)).
\end{array}
\end{eqnarray*}
Consider $\LGD$. If $s_1=0$ (and $s_0\neq 0$), $\LGD$ coincides with
$\LGD^\prime$ in 1., case $s_1 = 0$. Hence the same derivation and
conclusions apply.

Let now $s_1>0$, hence $\Ss=\Omega$.
Then $\sup(\LGD|\Omega)\geq\sup(\LGD|A)\geq0$.
The last inequality holds because
\begin{eqnarray}
\label{eq:inequality_holds}
\sup(\LGD|A)=s_1(1-\lpr(A))+s_0\lpr(X|A)+\sup(-s_0 X|A)
\end{eqnarray}
and from \eqref{eq:in_interval}, \eqref{eq:inequality_holds} we obtain:

if $s_0\geq 0$,
$\sup(\LGD|A)\geq s_1(1-\lpr(A))+s_0\inf(X|A)-s_0\inf(X|A)\geq 0$;

if $s_0<0$,
$\sup(\LGD|A)\geq s_1(1-\lpr(A))+s_0\sup(X|A)-s_0\sup(X|A)\geq 0$.

Referring to $\LGD^\prime$, if $s_0=0$ then $\Ss=A$,
$\sup(\LGD^\prime|A)=s_1(\sup(X|A)-\lpr(X|A))\geq 0$ by
\eqref{eq:in_interval}.

When $s_0\neq 0$, $\Ss=\Omega$ and
$\sup\LGD^\prime=\max\{\sup(\LGD^\prime|A),
\sup(\LGD^\prime|\nega{A})\}\geq 0$.

In fact, if $s_0>0$ then $\sup(\LGD^\prime|\nega{A})=s_0\lpr(A)>0$.
If $s_0<0$, $\sup(\LGD^\prime|A)=-s_1\lpr(X|A)-s_0(1-\lpr(A))+
s_1\sup(X|A)\geq -s_1\sup(X|A)-s_0(1-\lpr(A))+s_1\sup(X|A)\geq 0$.
\end{enumerate}
Conversely, let now $\lpr$ be $2$-coherent. Since $\lpr$ is also
$1$-coherent, $\lpr(X|A)$ satisfies condition
\eqref{eq:in_interval}. To see that also $r$ does so (the proof for
$q$ is identical), note that the gain
\begin{eqnarray*}
\LGD=s_1 (A-\lpr(A))-A(X-r), s_1>0
\end{eqnarray*}
is such that $\sup\LGD\geq 0$,
by \eqref{eq:cond_2-coherence}.
Since $\sup(\LGD|\nega{A})=-s_1\lpr(A)<0$,
necessarily $\sup(\LGD|A)=s_1(1-\lpr(A))+r-\inf(X|A)\geq 0$,
that is
\begin{eqnarray*}
\inf(X|A)\leq r+s_1(1-\lpr(A)), \forall s_1>0.
\end{eqnarray*}
From the above inequality, $r\geq\inf(X|A)$.

To prove that $r\leq\sup(X|A)$,
consider the gain
\begin{eqnarray*}
\LGD^\prime=A(X-r)-s_0(A-\lpr(A)), s_0<0,
\end{eqnarray*}
and note that $\LGD^\prime|\nega{A}=s_0\lpr(A)<0$. This implies, for
any $s_0<0$, $\sup(\LGD^\prime|A)=\sup(X|A)-r-s_0(1-\lpr(A))\geq 0$.
Hence, $r\leq\sup(X|A)$.
\end{proof}
\emph{Comment.}
Proposition \ref{pro:GBR_not_unique} establishes
that equation \eqref{eq:GBR_equation} has more solutions, when
$\lpr$ is $2$-coherent. Actually, there are infinitely many,
provided they comply with the internality condition
\eqref{eq:in_interval}. Lack of uniqueness means also that we are
not obliged to choose one of these solutions: any two of them can
$2$-coherently coexist in the set $\dset$ of Proposition
\ref{pro:GBR_not_unique}.
Even as many solutions as we wish may be found in a single $2$-coherent assessment.
Just think that this does not essentially alter
the proof of Proposition \ref{pro:GBR_not_unique}, since
$2$-coherence restricts checking it on gains referring to (at most)
couples of gambles.

Since a $2$-coherent prevision is also $2$-convex, it is clear that
the GBR will generally not be the unique solution of equation
\eqref{eq:GBR_equation} even when $\lpr$ is $2$-convex.
We omit detailing this case.

\section{Weakly consistent uncertainty models}
\label{sec:weakly_consistent_models} As mentioned in the
Introduction, a motivation for studying the loose forms of
consistency introduced in this paper is their capability of
encompassing or extending uncertainty models already investigated in
the literature. Even though these models may depart also
considerably from coherence and convexity, they can nevertheless be
accommodated into a unifying betting scheme, ranging from $2$-convex
to coherent lower previsions.
\subsection{Capacities and niveloids}
\label{sec:capacities_niveloids}
Focusing on $2$-convexity, we first recall a few definitions and
some results concerning unconditional $2$-convex lower previsions.
\begin{definition}
\label{def:capacity}
Given a finite partition $\prt$,
and denoting with $2^{\prt}$ its powerset,
a mapping $c:2^{\prt}\rightarrow [0,1]$ is a (normalised) \emph{capacity}
whenever $c(\varnothing)=0$, $c(\Omega)=1$ (normalisation) and
$\forall A_1, A_2 \in 2^{\prt}$ such that $A_1\Rightarrow A_2$, $c(A_1)\leq c(A_2)$ ($1$-monotonicity).
\end{definition}
\begin{definition}
\label{def:niveloid}
Given a linear space $\lset$ of random variables,
a \emph{niveloid} \cite{cer14, dol95} is a functional $N:\lset\rightarrow\overline{\rset}=\rset\cup\{-\infty,+\infty\}$
which is translation invariant and monotone, i.e. such that
\begin{eqnarray}
\label{eq:niveloids}
\begin{array}{lll}
N(X+\mu)=N(X)+\mu, \forall X\in\lset, \forall \mu\in\rset;\\
X\geq Y \mbox{ implies } N(X)\geq N(Y), \forall X,Y\in\lset.
\end{array}
\end{eqnarray}
\end{definition}
As well-known,
capacities are uncertainty measures with really minimal quantitative requirements.
Niveloids can be viewed as a generalisation of capacities to linear spaces of random variables
which preserves their minimality properties.
Strictly speaking,
this is true for centered niveloids,
i.e. such that $N(0)=0$.
In fact, the centering condition $N(0)=0 $ does not ensue from the definition of niveloid.
Note also that niveloids apply to random variables which may be unbounded too.

It has been proven in \cite[Section 4.1]{bar09}\footnote{ See
Footnote \ref{foo:1_become_2}. } that:
\begin{proposition}
\label{pro:cap_niv}
\begin{itemize}
\item[a)] Let $\lpr$ be defined on $2^{\prt}$.
Then $\lpr$ is a centered $2$-convex lower prevision if and only if
it is a capacity.
\item[b)] Let $\lpr$ be defined on a linear space $\lset$ of gambles.
Then $\lpr$ is a $2$-convex lower prevision if and only if it is a (finite-valued) niveloid.
\end{itemize}
\end{proposition}
Hence, an unconditional $2$-convex lower prevision is equivalent to
a capacity or a niveloid, on structured sets ($2^{\prt}$ or $\lset$
respectively). On non-structured sets,ù it extends these concepts.

$2$-convex conditional lower previsions are natural candidates to define conditional capacities and niveloids
on \emph{arbitrary} sets of, respectively, conditional events or gambles.
To the best of our knowledge,
such conditional versions have not been considered yet in this general conditional environment,
but rather in more specific cases.
For instance, \cite{cha01} focuses on updating rules for `convex' capacities,
which means for $2$-monotone lower probabilities,
while considering a single conditioning event.

Thus $2$-convex previsions may provide an appropriate framework for such extensions,
guaranteeing some minimal properties like the existence of a $2$-convex natural extension (when being centered).
Take for instance centered $2$-convex conditional lower probabilities.
They satisfy the properties one would require to a conditional capacity:
$\lpr(0|B)=0$, $\lpr(\Omega|B)=1$ (this follows from Proposition \ref{pro:C-convex_properties}, a)),
and $A|B\leq_{GN}C|D$ implies $\lpr(A|B)\leq\lpr(C|D)$ (Proposition \ref{pro:C-convex_properties}, c)).
Similarly,
centered $2$-convex lower previsions ensure generalisations of properties
\eqref{eq:niveloids}
(see especially Proposition \ref{2-convexity_structure}  and Remark \ref{rem:2-convexity} for the first property,
Proposition \ref{pro:C-convex_properties}, c) for the second).

\subsection{Value-at-Risk (VaR)}
\label{sec:VaR}
Several examples of weakly consistent models may be found among the many risk measures that have been proposed in the financial literature.
We shall discuss here Value-at-Risk (VaR),
probably the most widespread such measure.

A risk measure $\rho$ is a map $\rho:\dset\rightarrow\rset$ assigning a number
$\rho(X)$ to each gamble $X\in\dset$, aiming at measuring how `risky' $X$ is.
Risk measures are strongly connected to imprecise previsions:
any risk measure $\rho(X)$ on $X$ corresponds to the opposite $-\lpr(X)$ of a lower
prevision for $X$ \cite{pel03bis}.

Because of this correspondence,
we may transpose concepts developed for imprecise probability theory to risk measurement (and vice versa).
Hence it is possible to check whether a certain risk measure is coherent, convex,
or at least $2$-coherent or $2$-convex,
according to whether the corresponding $\lpr=-\rho$ is so.

As for VaR, it is essentially a quantile-based measure:
\begin{definition}[\cite{art99}]
\label{def:VaR}
Given a gamble $X$,
a probability $P$ on $\{(X\leq x): x\in\rset\}$,
and a real $\alpha\in ]0,1[$,
\emph{the Value-at-Risk of $X$ at level $\alpha$} is:\footnote{In alternative definitions of VaR, cf. \cite[Sec. 2.3.1]{den05},
the minus in \eqref{eq:VaR} is omitted (this corresponds to
reasoning in terms of losses) and/or the strict inequality in the $\inf$ is weak.
Their consistency properties are the same.}
\begin{eqnarray}
\label{eq:VaR}
VaR_\alpha(X)=-\inf\{x\in\rset:P(X\leq x)>\alpha\}.
\end{eqnarray}
\end{definition}
It is known that VaR is not coherent,
although it may be so under some additional, rather strong assumptions
\cite{pel03bis}.
Which are then its guaranteed consistency properties?
This amounts to investigating the consistency of a lower prevision
$\lpr^{V}_\alpha(X)=-VaR_\alpha(X)$,
by the correspondence mentioned above.
The next proposition ensures that \emph{VaR is centered $2$-convex},
while Example \ref{ex:varincoh} shows that it may even fail to be $2$-coherent.
\begin{proposition}
\label{pro:VaR_2_convex}
Let $\lset$ be a linear space of gambles, $\alpha\in]0,1[$
and $\lpr$ a probability on $\bigcup_{X\in\lset}\{(X\leq x): x\in\rset\}$.
Define $\lpr^V_\alpha$ as
\begin{eqnarray*}
\label{eq:PV}
\lpr^V_\alpha(X)=-VaR_\alpha(X)=\inf\{x\in\rset:P(X\leq x)>\alpha\}, \forall X\in\lset.
\end{eqnarray*}
Then $\lpr^V_\alpha$ is a centered $2$-convex lower prevision.
\end{proposition}
\begin{proof}
By Proposition \ref{pro:cap_niv} b), $\lpr^V_\alpha$ is $2$-convex iff it is a niveloid,
that is iff it is translation invariant and monotone.
Proving translation invariance and monotonicity is essentially the same as proving that $VaR_\alpha$ has these properties,
which is well known (cf. \cite[Sec. 2.3.2]{den05}).

As for centering, recalling \eqref{eq:VaR} we have $\lpr^V_\alpha(0)=-VaR_\alpha(0)=0$.
\end{proof}
\begin{example}[VaR may be $2$-incoherent]
\label{ex:varincoh}
Let $X$ be a $2$-valued gamble such that
$P(X=-1)=P(X=2)=0.5$.
Given $\alpha=0.6$,
it is easy to check using \eqref{eq:VaR} that
$\lpr^V_{0.6}(X)=-VaR_{0.6}(X)=\inf\{x:P(X\leq x)>0.6\}=2$,
while 
$\lpr^V_{0.6}(-X)=1$.
Hence
$\lpr^V_{0.6}(-X)>-\lpr^V_{0.6}(X)$,
meaning that axiom (A6) with $\lambda=-1$ does not hold for $\lpr^V_{0.6}$.
Since (A6) is a necessary condition for $2$-coherence,
$\lpr^V_{0.6}$ is not $2$-coherent.
\label{ex:VaR_not_2_coherent}
\end{example}
\begin{remark}[$2$-coherent models]
\label{rem:on_2_convex_models}
The models we have seen so far are $2$-convex.
$2$-coherence arises naturally with interval evaluations made up of
a lower $\lpr$ and an upper $\upr$
uncertainty measure, both $2$-convex,
like a capacity and its conjugate.
In fact, it is then natural to require that $\lpr\leq\upr$,
which is a follow-up of equation \eqref{eq:conjugate},
implied by $2$-coherence.
As another instance, we mention $p$-boxes.
While univariate $p$-boxes satisfy stronger consistency properties
(they correspond to a couple $(\lpr,\upr)$, where both $\lpr$, $\upr$ are precise probabilities),
bivariate $p$-boxes may be related with $2$-coherence
(cf. \cite[Sec. 3.1]{pel15bis}). 
\end{remark}

\section{Weak consistency in a desirability approach}
\label{sec:weak_consistency}
In this section we examine $2$-convexity and $2$-coherence from the viewpoint of desirability.
This is an alternative approach to rationality concepts for uncertainty measures going back to \cite{wil75}
in the case of conditional imprecise previsions.
It has been recently applied to a variety of other situations,
see e.g. the discussion in \cite{qua14} and the results in \cite{qua15}.

Roughly speaking, a set $\aset$ of gambles is considered.\footnote{
As will appear later, $\aset$ is included into some fixed linear
space of gambles.
}
It is such that its gambles are regarded as
\emph{desirable} or \emph{acceptable}. We may in general be willing
to establish some \emph{rationality criteria}, requiring that certain
gambles do, or do not, belong to $\aset$. The basic problem we shall
consider here is: which is the correspondence between the
rationality criteria we adopt and the consistency concepts of
centered $2$-convexity or alternatively $2$-coherence? More
specifically, the following two questions arise:
\begin{itemize}
\item[Q1)] Which rationality criteria should be required to the elements of
a set $\aset$,
so that a conditional lower prevision $\lpr$ may be obtained from $\aset$ that is $2$-coherent
(alternatively,  $2$-convex)?
\item[Q2)] Conversely,
given a $2$-coherent (alternatively, $2$-convex) $\lpr$,
does it determine a set $\asetp$ with certain rationality properties?
\end{itemize}
In the case that $\lpr$ is coherent, the answer to Q1) and Q2) was
given by Williams in \cite{wil75}. Our approach to solving Q1) and
Q2) was largely influenced by his work. Preliminarily, some notation
must be introduced.
\begin{definition}
\label{def:desirability_sets}
Let $\xset$ be a linear space of gambles,
$\bset\subset\xset$ a set of (indicators of) events,
$\bsetp=\bset\setminus\{\varnothing\}$.
We suppose $\Omega\in\bset$ and $BX\in\xset$, $\forall B\in\bset$, $\forall X\in\xset$.\footnote{
Note that if $X\in\xset$ and $B\in\bsetp$, $X|B\in\dlin$ in the notation of the preceding sections.
}
Define then
\begin{eqnarray*}
\begin{array}{lll}
\xsetg=\{X\in\xset:\inf X\geq 0\},\\
\xsetl=\{X\in\xset:\sup X\leq 0\},
\end{array}
\end{eqnarray*}
and, $\forall B\in\bset$,
\begin{eqnarray*}
\begin{array}{lll}
\rb=\{X\in\xset:BX=X\},\\
\rbg=\{X\in\rb:\inf\{X|B\}>0\},\\
\rbl=\{X\in\rb:\sup\{X|B\}<0\}.
\end{array}
\end{eqnarray*}
If $\mathcal{S}$ and $\mathcal{T}$ are subsets of $\xset$, their
\emph{Minkowski sum} is
\begin{eqnarray*}
\mathcal{S}+\mathcal{T}=\{X+Y:X\in\mathcal{S}, Y\in\mathcal{T}\}.
\end{eqnarray*}
We shall use similar compact notation later.
For instance,
$\lambda\mathcal{S}+\mu\mathcal{T}\subseteq\mathcal{U}$, $\forall\lambda,\mu\geq 0$,
means:
$\forall X\in\mathcal{S}$, $\forall Y\in\mathcal{T}$, $\forall\lambda,\mu\geq 0$,
$\lambda X+\mu Y\in\mathcal{U}$.
\end{definition}
\begin{lemma}
\label{lem:rb_properties}
Properties of the sets $\rb$:
\begin{itemize}
\item[a)]
$\forall B$, $\rb$ is a linear space.
\item[b)]
If $X\in\rb$ and $B\Rightarrow A$, then $X\in\ra$.
\end{itemize}
\end{lemma}
\begin{proof}
a) is trivial.
As for b),
we have $AX=I_A X=(I_B + I_{A\wedge\nega{B}})X=I_B X=X$.
\end{proof}

\subsection{Desirability axioms for $2$-coherent previsions}
\label{sec:desirability_2_coherence}
The following proposition answers question Q1) completely for $2$-coherence:
\begin{proposition}
\label{pro:lower_from_set}
Let $\aset\subseteq\xset$ be such that
\begin{itemize}
\item[a)] $\lambda\aset+\rbg\subseteq\aset$, $\forall\lambda\geq 0$, $\forall B\in\bset$;
\item[b)] $\rbl\cap\aset=\emptyset$, $\forall B\in\bset$.
\item[c)]
$(\rbu\cap\aset)+(\rbd\cap\aset)\subseteq\\
\hspace{1.00cm}\rbud\setminus\rbud^{\prec}, \forall B_1, B_2\in\bset.$
\end{itemize}
Define, $\forall X|B\in\dlin$,
\begin{eqnarray}
\label{eq:lower_from_set}
\lpr(X|B)=&\sup\{x\in\rset:B(X-x)\in\aset\}.
\end{eqnarray}
Then, $\lpr$ is $2$-coherent on $\dlin$.
\end{proposition}
\begin{proof}
The core idea of the proof is to show that,
for any given $X_0|B_0, X_1|B_1\in\dlin$,
$\forall s_1\geq 0$, $\forall s_0\in\rset$,
$\lpr$ defined by \eqref{eq:lower_from_set} is such that $\LGD$
satisfies condition \eqref{eq:cond_2-coherence} in Definition \ref{def:cond_2-coherence},
and therefore $\lpr$ is $2$-coherent.

For this, define first the following gambles:
\begin{eqnarray}
\label{eq:k}
K=\sup\Biggl(\sum_{\substack{i=0 \\ s_i\neq 0}}^{1} B_i|\Ss\Biggr),
\end{eqnarray}
\begin{eqnarray}
\label{eq:s}
S_i=\left\{
    \begin{array}{ll}
     |s_i| B_i(X_i-\lpr(X_i|B_i))+\frac{s_i \epsilon}{|s_i| K} B_i \mbox{ if } s_i\neq 0\\
0 \hfill \mbox{ if } s_i=0
\end{array}
\right..
\end{eqnarray}
In equation \eqref{eq:s}, $i=0,1$ and $\epsilon>0$.

Note that $K>0$, as it can take values in $\{1,2\}$.
We analyse now the relationships between $S_0$, $S_1$ in \eqref{eq:s} and $\aset$.
The following facts will be used later in the proof,
when expressing $\LGD$ in terms of $S_0$, $S_1$ and $K$.
\begin{itemize}
\item[i)]If $s_i>0$,
$S_i=s_i B_i(X_i-\lpr(X_i|B_i))+\frac{\epsilon}{K} B_i \in\aset$, $i=0,1$.

In fact,
by the definition of $\lpr$ in \eqref{eq:lower_from_set},
$\exists t\in [0,\frac{\epsilon}{s_i K}[$ such that
$B_i (X_i-(\lpr(X_i|B_i)-t))\in\aset$.

Writing then
\begin{eqnarray*}
\begin{array}{ll}
     S_i=s_i B_i(X_i-(\lpr(X_i|B_i)-\frac{\epsilon}{s_i K}))= \\
     s_i B_i(X_i - (\lpr(X_i|B_i)-t))+s_i B_i (\frac{\epsilon}{s_i K}-t),
\end{array}
\end{eqnarray*}
we note that the second term in the summation,
$s_i B_i (\frac{\epsilon}{s_i K}-t)$,
belongs to $\rbgi$.
Since the first term is in $\aset$,
$S_i\in\aset$ by assumption a).
\item[ii)]
If $s_o<0$,
$S_0=-s_0 B_0 (X_0-\lpr(X_0|B_0))-\frac{\epsilon}{K}B_0\notin\aset$.

Suppose by contradiction
$S_0=-s_0 B_0 (X_0-(\lpr(X_0|B_0)+\frac{\epsilon}{-s_0 K}))\in\aset$.
Since $\delta B_0\in{\mathcal{R}(B_{0})^{\succ}}$, $\forall\delta>0$,
assumption a) gives
\begin{eqnarray}
\label{eq:s0_with_delta}
\frac{S_0}{-s_0}+\delta B_0=B_0(X_0-(\lpr(X_0|B_0)+\frac{\epsilon}{-s_0 K}-\delta))\in\aset.
\end{eqnarray}
Taking $0<\delta<\frac{\epsilon}{-s_0 K}$,
it is $\lpr(X_0|B_0)+\frac{\epsilon}{-s_0 K}-\delta>\lpr(X_0|B_0)$,
so that \eqref{eq:s0_with_delta} contradicts the definition of $\lpr$ in \eqref{eq:lower_from_set}.
\end{itemize}
The gain $\LGD$ is a function of the gambles $S_0$, $S_1$ and $K$:
\begin{eqnarray}
\label{eq:lgd}
\LGD=\left\{
    \begin{array}{ll}
       S_1+S_0-\frac{\epsilon}{K}\sum_{\substack{i=0 \\ s_i\neq 0}}^{1} B_i \mbox{ if } s_0\geq 0\\
       S_1-S_0-\frac{\epsilon}{K}\sum_{\substack{i=0 \\ s_i\neq 0}}^{1} B_i \mbox{ if } s_0\leq 0.
\end{array}
\right.
\end{eqnarray}
Referring to this representation,
define $T=S_1+S_0$ if $s_0\geq 0$, $T=S_1-S_0$, if $s_0\leq 0$.
We prove that $\sup(T|\Ss)\geq 0$,
distinguishing three cases:
\begin{itemize}
\item[$\bullet$]
$s_1>0$, $s_0>0$.

By i), $S_i\in\aset$, $i=0,1$.
It is also $S_i\in\rbi$,
hence $S_i\in\rbi\cap\aset$, $i=0,1$.
Using assumption c) we deduce
$S_1+S_0\in{\mathcal{R}(B_{0}\vee B_{1})}\setminus{\mathcal{R}(B_{0}\vee B_{1})^{\prec}}$,
which means that $\sup(S_1+S_0|\Ss)\geq 0$.
\item[$\bullet$]
$s_0<0$.

By contradiction,
suppose $\sup(S_1-S_0|\Ss)=\sup(T|\Ss)<0$, and hence $\inf(-T|\Ss)>0$.
If $s_1=0$, this means $-T=S_0\in{\mathcal{R}(B_{0})^{\succ}}$.
If $s_1>0$,
since then $-T\in{\mathcal{R}(B_{0}\vee B_{1})}$ by Lemma \ref{lem:rb_properties},
$\inf(-T|\Ss)>0$ implies $-T\in{\mathcal{R}(B_{0}\vee B_{1})^{\succ}}$.
In both instances,
assumption a) can be applied (with $\lambda = 0$ if $s_1=0$,
recalling instead that
$S_1\in\aset$ by i) if $s_1>0$)
to deduce $S_0\in\aset$,
which contradicts ii).
%
%
%
\item[$\bullet$]
$s_1>0$, $s_0=0$ or $s_1=0$, $s_0>0$.

If $s_1>0$, $s_0=0$,
then using i) $T=S_1\in\aset\cap\rbu$.
By assumption b),
$\sup(T|\Ss)=\sup(S_1|B_1)\geq 0$.

If $s_1=0$, $s_0>0$,
the argument is analogous.

\end{itemize}
Recalling \eqref{eq:lgd},
$T=\LGD+\frac{\epsilon}{K}\sum_{\substack{i=0 \\ s_i\neq 0}}^{1}B_i$.
Since $\sup(T|\Ss)\geq 0$ and using the definition of $K$ at the next equality,
we obtain
\begin{eqnarray*}
\label{eq:proof_1_1}
\begin{array}{lll}
0\leq\sup(\LGD+\frac{\epsilon}{K}\sum_{\substack{i=0 \\ s_i\neq 0}}^{1}B_i|\Ss) \\
\leq\sup(\LGD|\Ss)+\frac{\epsilon}{K}\sup(\sum_{\substack{i=0 \\ s_i\neq 0}}^{1}B_i|\Ss)\\
=\sup(\LGD|\Ss)+\epsilon.
\end{array}
\end{eqnarray*}
Thus,
$\sup(\LGD|\Ss)\geq -\epsilon$, $\forall\epsilon>0$,
that is $\sup(\LGD|\Ss)\geq 0$.
\end{proof}
Unlike the case of coherent conditional lower previsions examined in \cite[Section 3.1]{wil75},
$\aset$ does not need to be a cone in Proposition~\ref{pro:lower_from_set}:
given $X, Y\in\aset$, $\lambda\geq 0$, neither $X+Y$ nor $\lambda X$ are guaranteed to belong to $\aset$.
Actually, condition a) represents a weakening of the cone axioms:
if $X\in\aset$,
$Y\in\rbg$ and $\lambda\geq 0$, then $\lambda X+Y\in\aset$.
This implies also
$\rbg\subseteq\aset\ \forall B\in\bset$,
a condition that, like also b), is required for coherence as well
(see (C1'), (C2') in \cite[Section 3.1]{wil75}).

The interpretation of b) is that of an \emph{avoiding partial loss}
condition:
we can expect no gain from owning a gamble in $\rbl$,
when $B$ is true,
therefore such gambles cannot be included into $\aset$.

As for c), writing it in an extended form,
it tells us that:
if $X_1, X_2\in\aset$, $B_1 X_1=X_1$, $B_2 X_2=X_2$,
then
$(B_1\vee B_2)(X_1 + X_2)=X_1 + X_2$
and $\sup(X_1 + X_2|B_1\vee B_2)\geq 0$.
Note that if $X_1\in\rbu$ and $X_2\in\rbd$,
it always holds that $X_1+X_2\in\rbud$
by Lemma \ref{lem:rb_properties},
without having to impose it by means of axiom c).

Therefore,
the essential condition in axiom c) is that
if $X_1$, $X_2$ are desirable (belonging to $\aset$),
this does not imply that $X_1 + X_2\in\aset$
(which is required for coherence in \cite{wil75, wil07}),
but only that $X_1 + X_2$ is not necessarily discarded
by resorting to b) with $B=B_1\vee B_2$.
To illustrate this concept, let for instance $B_1=B_2=\Omega$ in c),
so that $\rbu=\rbd=\rbud=\mathcal{R}(\Omega)=\xset$.
Then, c) implies $X_1+X_2\notin{\mathcal{R}(\Omega)}^{\prec}$,
making impossible to apply b) in order to discard $X_1+X_2$ from $\aset$.

As for question Q2),
an answer is given by the following proposition, when $\lpr$ is $2$-coherent.

\begin{proposition}
\label{pro:set_from_lower}
Let $\lpr:\dlin\rightarrow\rset$ be $2$-coherent.
Define
\begin{eqnarray*}
\label{eq:set_from_lower}
\begin{array}{lll}
\asetp=\{\lambda B(X-x)+Y:X|B\in\dlin, x<\lpr(X|B), Y\in\xsetg,\lambda\geq 0\}.
\end{array}
\end{eqnarray*}
Then the set $\asetp$ is such that:
\begin{itemize}
\item[a')] $a\asetp+\xsetg\subseteq\asetp$, $\forall a\geq 0$;
\item[b')] $\xsetl\cap\asetp=\{0\}$;
\item[c')] $(\asetp + \asetp)\setminus\{0\}\subseteq\xset\setminus\xsetl$;
\item[d')] $\lpr(X|B)=\sup\{x\in\rset:B(X-x)\in\asetp\}$, $\forall X|B\in\dlin$.
\end{itemize}
\end{proposition}
\begin{proof}
\emph{Proof of a').}
Take $Z_1=\lambda B (X-x)+Y\in\asetp$,
$a\geq 0$ and $Z_2\in\xsetg$.
Then,
$aZ_1 + Z_2=a\lambda B(X-x)+aY+Z_2$,
with $a\lambda\geq 0$, $X\in\xset$, $B\in\bsetp$,
$x\leq\lpr(X|B)$ and $aY+Z_2\in\xsetg$
(because $\inf(aY+Z_2)\geq a\inf Y+\inf Z_2\geq 0$).
Therefore $aZ_1 + Z_2\in\asetp$.

\emph{Proof of b').}
Take $Z=\lambda B(X-x)+Y\in\asetp$.

\emph{If} $\lambda=0$,
$Z=Y\in\xsetg$.
Therefore $\sup Z\geq\inf Z\geq 0$,
so that $Z\notin\xsetl$ if $Y\neq 0$,
while $Z\in\xsetl$ if and only if $Y=0$.

\emph{If} $\lambda>0$,
$Z\geq\lambda B(X-x)$
because $Y\geq\inf Y\geq 0$.
It follows that
\begin{eqnarray*}
\label{eq:sequence_lambda}
\begin{array}{lll}
\sup (Z|B)\geq \sup(\lambda B(X-x)|B)=\\
\sup(\lambda B(X-\lpr(X|B))+\lambda B(\lpr(X|B)-x)|B)\geq\\
\sup(\lambda B(X-\lpr(X|B))|B)+\inf(\lambda B(\lpr(X|B)-x)|B)>0
\end{array}
\end{eqnarray*}
using the property
\begin{eqnarray}
\label{eq:sup_notless_inf_sup}
\sup(X_1 + X_2)\geq\inf X_1 + \sup X_2
\end{eqnarray}
with
$X_1=\lambda B(\lpr(X|B)-x)$,
$X_2=\lambda B(X-\lpr(X|B))$
at the second inequality;
the final inequality follows from
$\sup(\lambda B(X-\lpr(X|B))|B)\geq 0$
by $2$-coherence of $\lpr$
(equation \eqref{eq:cond_2-coherence} with $s_0=0$)
and from $\inf(\lambda B(\lpr(X|B)-x)|B)>0$ since $\lambda>0$,
$\lpr(X|B)>x$.

The above derivation ensures $\sup Z\geq \sup Z|B>0$,
i.e. $Z\notin\xsetl$.

Whatever is $\lambda\geq 0$ then,
$Z\in\asetp$ implies either $Z\notin\xsetl$ or $Z=0$.
Therefore,
since $0\in\asetp\cap\xsetl$, we have that $\asetp\cap\xsetl=\{0\}$.

\emph{Proof of c')}
To establish c'),
we prove that for any $Z_1, Z_2 \in\asetp$,
$Z_1 + Z_2 \neq 0$,
it holds that $\sup(Z_1 + Z_2)>0$.

From the definition of $\asetp$,
we have that
\begin{eqnarray*}
Z_i=\lambda_i B_i (X_i - x_i)+Y_i\ (i=1,2).
\end{eqnarray*}
If $\lambda_1=\lambda_2=0$,
then $Z_i=Y_i \in\xsetg$, $i=1,2$,
and at least one of $Y_1$, $Y_2$ is not zero
(because $Z_1+Z_2\neq 0$).
If for instance $Y_1\neq 0$,
then $\sup Y_1>0$ (because $Y_1\in\xsetg$).
It follows that $\sup(Z_1+Z_2)=\sup(Y_1+Y_2)\geq\sup Y_1+\inf Y_2\geq\sup Y_1>0$,
where the first inequality follows from \eqref{eq:sup_notless_inf_sup}.

We may therefore suppose $\lambda_1+\lambda_2>0$ in the sequel of the proof,
defining $\slambda=\vee\{B_i:\lambda_i>0, i=1,2\}(\neq\varnothing)$.
We have that
\begin{eqnarray*}
\label{eq:seq_sl}
\begin{array}{ll}
Z_1+Z_2|\slambda&=\sum_{i=1}^{2}\lambda_i B_i (X_i-x_i)|\slambda+(Y_1+Y_2)|\slambda \\
&=\sum_{i=1}^{2}\lambda_i B_i (X_i - \lpr(X_i|B_i))|\slambda \\
&+\sum_{i=1}^{2}\lambda_i B_i (\lpr(X_i|B_i)-x_i)|\slambda+(Y_1 + Y_2)|\slambda \\
&\geq\sum_{i=1}^{2}\lambda_i B_i (X_i-\lpr(X_i|B_i))|\slambda + \delta + \inf Y_1 + \inf Y_2,
\end{array}
\end{eqnarray*}
where $\delta=\min\sum_{i=1}^{2}\lambda_i B_i(\lpr(X_i|B_i)-x_i)|\slambda$.
Recalling that $\lambda_i\geq 0$,
$\lpr(X_i|B_i)\\ >x_i$, $\lambda_1+\lambda_2>0$,
it is easy to realise that $\delta>0$.

Using this fact in the strict inequality of the following derivation,
we obtain
\begin{eqnarray*}
\label{eq:seq_two_z}
\begin{array}{ll}
\sup(Z_1+Z_2)\geq\sup(Z_1+Z_2|\slambda)\geq \\
\sup(\sum_{i=1}^{2}\lambda_i B_i (X_i-\lpr(X_i|B_i))|\slambda) + \delta + \inf Y_1 + \inf Y_2> \\
\sup(\sum_{i=1}^{2}\lambda_i B_i (X_i-\lpr(X_i|B_i))|\slambda)\geq 0,
\end{array}
\end{eqnarray*}
the final inequality holding because $\lpr$ is $2$-coherent.

\emph{Proof of d')}
Let $S=\{x: B(X-x)\in\asetp\}$.

In the first (and larger) part of the proof,
we shall prove that
\begin{eqnarray}
\label{eq:Xgx}
\lpr(X|B)\geq x, \forall x\in S.
\end{eqnarray}
For this, let $\overline{x}\in S$.
Therefore,
\begin{eqnarray}
\label{eq:EXl}
B(X-\overline{x})=\lambda A(Z-z)+Y,
\end{eqnarray}
with $\lambda\geq 0$, $A\in\bsetp$, $Z\in\xset$, $z<\lpr(Z|A)$, $Y\in\xsetg$.
We distinguish the cases $\lambda = 0$ and $\lambda>0$.
\begin{itemize}
\item[$\bullet$]
$\lambda = 0$.

From $\inf(X|B)-\overline{x}=\inf (B(X-\overline{x})|B)\geq\inf(B(X-\overline{x}))=
\inf Y\geq 0$
(the last inequality holding because $Y\in\xsetg$),
we obtain $\overline{x}\leq\inf(X|B)$.
Therefore also $\overline{x}\leq\lpr(X|B)$,
because by $2$-coherence $\inf(X|B)\leq\lpr(X|B)$.
(Actually,
centered $2$-convexity is enough for this, by Proposition \ref{pro:C-convex_properties}), a).)
\item[$\bullet$]
$\lambda>0$.

From \eqref{eq:EXl},
$B(X-\overline{x})\geq\lambda A(Z-z)$,
hence
\begin{eqnarray}
\label{eq:supAZ}
\sup(\lambda A(Z-z)-B(X-\overline{x}))\leq 0.
\end{eqnarray}
Define now
\begin{eqnarray*}
\label{eq:X1_X2}
\begin{array}{ll}
X_1&=\lambda A(z-\lpr(Z|A))-B(\overline{x}-\lpr(X|B)) \\
X_2&=\lambda A(Z-\lpr(Z|A))-B(X-\lpr(X|B))-X_1 \\
&=\lambda A(Z-z)-B(X-\overline{x}).
\end{array}
\end{eqnarray*}
Observe that:
\begin{itemize}
\item[i)]
$\sup(X_1 + X_2|A\vee B)\geq 0$.

Since $X_1 + X_2 = \lambda A(Z-\lpr(Z|A))-B(X-\lpr(X|B))$,
this follows from $2$-coherence of $\lpr$
(equation \eqref{eq:cond_2-coherence}, with $s_1=\lambda$, $s_0=1$).
\item[ii)]
$\sup(X_2|A\vee B)\leq\sup(X_2)\leq 0$.

In fact, $X_2$ is the argument of the supremum in equation \eqref{eq:supAZ}.
\end{itemize}
Using i) and ii), we obtain
\begin{eqnarray*}
\begin{array}{ll}
\sup(X_1|A\vee B)&\geq\sup(X_1 + X_2|A\vee B)-\sup(X_2|A\vee B) \\
&\geq-\sup(X_2|A\vee B)\geq 0.
\end{array}
\end{eqnarray*}
Now we know that
$\sup(X_1|A\vee B)\geq 0$.
On the other hand,
$X_1|A\vee B$ is a three-valued gamble (at most),
and precisely it takes the following values
\begin{align*}
&\lambda(z-\lpr(Z|A))-(\overline{x}-\lpr(X|B)) & &\mbox{ on } A\wedge B, \mbox{ when } A\wedge B\neq\varnothing; \\
&-(\overline{x}-\lpr(X|B)) & &\mbox{ on } \nega{A}\wedge B, \mbox{ when } \nega{A}\wedge B\neq\varnothing; \\
&\lambda(z-\lpr(Z|A))<0 & &\mbox{ on } A\wedge\nega{B}, \mbox{ when } A\wedge\nega{B}\neq\varnothing.
\end{align*}
Therefore,
$-(\overline{x}-\lpr(X|B))\geq\sup(X_1|A\vee B)\geq0,$\footnote{
The first inequality can be strict if $\nega{A}\wedge B=\varnothing$.
Note that $\nega{A}\wedge B=A\wedge B=\varnothing$ cannot occur,
since it implies $B=\varnothing$.
}
i.e. $\overline{x}\leq\lpr(X|B)$.
\end{itemize}
Thus \eqref{eq:Xgx} holds.
It remains to observe that $\forall x<\lpr(X|B)$,
it is $B(X-x)\in\asetp$,
by definition of $\asetp$ and since $0\in\xsetg$.
This means that $x\in S$.
Consequently
\begin{eqnarray*}
\lpr(X|B)=\sup S=\sup\{x\in\rset:B(X-x)\in\asetp)\}.
\end{eqnarray*}
\end{proof}
Proposition~\ref{pro:set_from_lower} states the existence
of a set of desirable gambles $\asetp$,
in accordance with a given $2$-coherent conditional lower prevision $\lpr$
and satisfying the rationality criteria a'), b'), c').
Comparing a'), b') with a), b) in Proposition \ref{pro:lower_from_set},
a clear similarity comes evident:
essentially, the sets $\rbg$, $\rbl$, $B\in\bset$, have been replaced with $\xsetg$,
$\xsetl$ respectively.
As a consequence, note that $0\in\asetp$.

The interpretation of c') is similar to c) in Proposition~\ref{pro:lower_from_set}.
It tells that:
if $X_1, X_2\in\asetp$, $X_1+X_2\neq 0$,
then $\sup(X_1 + X_2)>0$.
Again,
coherence would allow the stronger implication
$X_1, X_2\in\asetp\to X_1 + X_2\in\asetp$,
while $2$-coherence only ensures that $X_1+X_2$ does not belong to the (near) rejection set $\xsetl$.

Actually, a'), b') c') prove to be stronger than a), b), c).
This means that any $2$-coherent conditional prevision can be represented
through a set of desirable gambles $\asetp$
satisfying the necessary axioms a'), b'), c'),
but also that, at the same time, $\asetp$ satisfies the weaker axioms
a), b), c) in Proposition \ref{pro:lower_from_set}.

\subsection{Desirability axioms for $2$-convex previsions}
\label{sec:desirability_2_convex}
A comparison between \eqref{eq:cond_2-convexity} in Definition \ref{def:2-convexity} and \eqref{eq:cond_2-coherence} in Definition \ref{def:cond_2-coherence}
intuitively suggests that we can get an answer to Q1) for $2$-convexity
from a reduced form of Proposition \ref{pro:lower_from_set}, with $\lambda=1$.
More precisely, the following proposition holds:
\begin{proposition}
\label{pro:lower_from_setc}
Let $\aset\subseteq\xset$ be such that
\begin{itemize}
\item[a)] $\aset+\rbg\subseteq\aset$, $\forall B\in\bset$;
\item[b)] $\rbl\cap\aset=\emptyset$, $\forall B\in\bset$.
\end{itemize}
Define, $\forall X|B\in\dlin$,
\begin{eqnarray}
\label{eq:lower_from_setc}
\lpr(X|B)=&\sup\{x\in\rset:B(X-x)\in\aset\}.
\end{eqnarray}
Then,
\begin{itemize}
\item[1)] $\lpr$ is $2$-convex on $\dlin$;
\item[2)] $\lpr$ is centered iff $\rbg\subseteq\aset\ \forall B\in\bset$.
\end{itemize}
\end{proposition}
\begin{proof}
\emph{Proof of 1)}
The proof is a simplification of that of Proposition \ref{pro:lower_from_set}.
Analogously,
it is checked that condition \eqref{eq:cond_2-convexity} in
Definition \ref{def:2-convexity} is satisfied for $\LGDC$,
where $\lpr$ is defined by \eqref{eq:lower_from_setc}.
The same steps are followed:
first, the definitions of $K$, $S_0$, $S_1$ in \eqref{eq:k}, \eqref{eq:s} simplify to
\begin{eqnarray*}
\label{eq:kc}
\begin{array}{lll}
K=\sup(B_0 + B_1|B_0 \vee B_1), \\
S_0 = B_0 (X_0-\lpr(X_0|B_0))-\frac{\epsilon}{K} B_0, \\
S_1 = B_1 (X_1-\lpr(X_1|B_1))+\frac{\epsilon}{K} B_1.
\end{array}
\end{eqnarray*}
Then,
the following are proven in the same way:
\begin{center}
i) $S_1\in\aset$;\quad\quad\quad ii) $S_0\notin\aset$.
\end{center}
Equation \eqref{eq:lgd} reduces here to
\begin{eqnarray*}
\label{eq:lgdc}
\LGDC=S_1-S_0-\frac{\epsilon}{K}(B_0 + B_1)
\end{eqnarray*}
and defining $T=S_1 - S_0$,
it is $\sup(T|B_0\vee B_1)\geq 0$
(see the case $(s_0<0, s_1>0)$ in the proof of Proposition \ref{pro:lower_from_set}).
This fact is exploited to show that $\sup(\LGDC|B_0 \vee B_1)\geq 0$,
with the same computations of the final part in the proof of Proposition \ref{pro:lower_from_set}.

\emph{Proof of 2)}
Suppose $\rbg\subseteq\aset$.
We prove that then $\lpr$ is centered.
In fact, by \eqref{eq:lower_from_setc}
\begin{eqnarray*}
\label{eq:lprzero}
\lpr(0|B)=\sup\{x:-Bx\in\aset\}, \forall 0|B\in\dlin.
\end{eqnarray*}

For $x<0$,
$\inf(-Bx|B)=-x>0$,
so that $-Bx\in\rbg\subseteq\aset$.

For $x>0$,
$\sup(-Bx|B)=-x<0$,
which implies $-Bx\notin\aset$ by property b).

Therefore $\sup\{x:-Bx\in\aset\}=0$,
i.e. $\lpr(0|B)=0$.

Conversely, suppose now
\begin{eqnarray*}
\label{eq:lprzerobis}
\lpr(0|B)=\sup\{x:-Bx\in\aset\}=0, \forall 0|B\in\dlin.
\end{eqnarray*}
We prove that $\rbg\subseteq\aset$ in two steps.
\begin{itemize}
\item[i)]
$-Bx\in\aset, \forall x<0$.

To see this,
take $\overline{x}<0$.
By definition of supremum,
$\exists\widetilde{x}:\overline{x}<\widetilde{x}\leq 0$,
$-B\widetilde{x}\in\aset$.
Writing $-B\overline{x}=-B\widetilde{x}+B(\widetilde{x}-\overline{x})$,
it is $B(\widetilde{x}-\overline{x})\in\rbg$,
because $\widetilde{x}-\overline{x}>0$.
By property a),
$-Bx\in\aset+\rbg\subseteq\aset$,
that is $-Bx\in\aset$.
\item[ii)]
$\rbg\subseteq\aset, \forall B\in\bsetp$.

For the proof,
let $X\in\rbg$.
This implies $\inf(X|B)>0$,
so that $\delta$ can be chosen,
such that $0<\delta<\inf(X|B)$.
Then
\begin{eqnarray}
\label{eq:bx}
X=BX=B(X-\inf(X|B)+\delta)-B(\delta-\inf(X|B)).
\end{eqnarray}
Since $\delta-\inf(X|B)<0$,
it is $-B(\delta-\inf(X|B))\in\aset$, by i).

Since $\inf(B(X-\inf(X|B)+\delta)|B)=\inf(X|B)-\inf(X|B)+\delta>0$,
it holds that $B(X-\inf(X|B)+\delta)\in\rbg$.

Applying axiom a) to the decomposition \eqref{eq:bx},
it ensues that $X\in\aset$,
that is $\rbg\subseteq\aset$.
\end{itemize}
\end{proof}
An analogously reduced form of Proposition \ref{pro:set_from_lower}
allows us to answer question Q2) for $2$-convexity.
\begin{proposition}
\label{pro:set_from_lowerc}
Let $\lpr:\dlin\rightarrow\rset$ be $2$-convex.
Define
\begin{eqnarray*}
\label{eq:set_from_lowerc}
\begin{array}{lll}
\asetp=\{B(X-x)+Y:X|B\in\dlin, x<\lpr(X|B), Y\in\xsetg\}.
\end{array}
\end{eqnarray*}
\begin{itemize}
\item[1)] The set $\asetp$ is such that:
\begin{itemize}
\item[a)] $\asetp+\xsetg\subseteq\asetp$;
\item[b)] $\asetp\cap\xsetl=\emptyset$ iff $\lpr$ is $1$-AUL;
\item[c)] $\lpr(X|B)=\sup\{x\in\rset:B(X-x)\in\asetp\}$, $\forall X|B\in\dlin$.
\end{itemize}
\item[2)] If $\lpr$ is centered, then $\rbg\subseteq\asetp\ \forall B\in\bset$;
if $\lpr$ is $1$-AUL and $\rbg\subseteq\asetp\ \forall B\in\bset$, then $\lpr$ is centered.
\end{itemize}
\end{proposition}
\begin{proof}
\emph{Proof of 1).}
Apart from the converse implication in b),
the proof is a simplified version of the proof of Proposition \ref{pro:set_from_lower}.
Precisely,
\begin{itemize}
\item[$\bullet$]
\emph{Proof of a).}
See proof of a') in Proposition \ref{pro:set_from_lower},
with $\lambda=1$.
\item[$\bullet$]
\emph{Proof of b).}
If $\lpr$ is $1$-AUL, then $\asetp\cap\xsetl=\emptyset$  follows from the proof of b') in Proposition \ref{pro:set_from_lower},
taking $\lambda=1$;
when proving that $\sup(Z|B)>0$,
the step resorting to $2$-coherence uses now $1$-AUL to justify by \eqref{eq:1_AUL} that
$\sup(B(X-\lpr(X|B))|B)\geq 0$.

We prove now the converse implication,
that if $\asetp\cap\xsetl=\emptyset$ then $\lpr$ is $1$-AUL.
Suppose $\asetp\cap\xsetl=\emptyset$ while $\lpr$ is not $1$-AUL,
which means that there exists $X|B$ such that $\sup(X|B)<\lpr(X|B)$.
Then $Z=B(X-\sup(X|B))\in\asetp$,
because $0\in\xsetg$.
Since $\sup Z=\ \max(\sup(Z|\nega{B}),\ \sup(Z|B))=\ \max(0, \sup(X|B)-\sup(X|B))=0$,
it is also $Z\in\xsetl$ and therefore $Z\in\asetp\cap\xsetl$,
contradicting the assumption $\asetp\cap\xsetl=\emptyset$.
\item[$\bullet$]
\emph{Proof of c).}
Special case of the proof of d') in Proposition \ref{pro:set_from_lower}
(put $\lambda=1$ and derive i) from $2$-convexity rather than $2$-coherence).
\end{itemize}
\emph{Proof of 2).}
From c),
we may write
\begin{eqnarray}
\label{eq;lpzero}
\lpr(0|B)=\sup\{x\in\rset:-Bx\in\asetp\}, \forall 0|B\in\dlin.
\end{eqnarray}
We prove that \emph{if $\lpr$ is centered then $\rbg\subseteq\asetp$}.

Suppose then $\lpr$ centered, which means by \eqref{eq;lpzero}
\begin{eqnarray*}
\lpr(0|B)=\sup\{x:-Bx\in\asetp\}=0, \forall B\in\bsetp.
\end{eqnarray*}
Let us first prove that
\begin{eqnarray}
\label{eq:minusBxinAprime}
-Bx\in\asetp, \forall x<0.
\end{eqnarray}
In fact,
let $\overline{x}<0$.
By the definition of supremum,
$\exists \widetilde{x}: \overline{x}<\widetilde{x}\leq 0$
and $-B\widetilde{x}\in\asetp$.
Hence
$-B\overline{x}=-B\widetilde{x}+B(\widetilde{x}-\overline{x})\in\asetp$ by property a),
given that $B(\widetilde{x}-\overline{x})\in\xsetg$,
since $\inf (B(\widetilde{x}-\overline{x}))=\min(0,\widetilde{x}-\overline{x})=0$.

Now let $X\in\rbg$, $\delta: 0<\delta<\inf(X|B)$.
Writing $X=B(X-\inf(X|B)+\delta)-B(\delta-\inf(X|B))$,
it holds that $-B(\delta-\inf(X|B))\in\asetp$,
using \eqref{eq:minusBxinAprime},
and that $B(X-\inf(X|B)+\delta)\in\xsetg$
because $\inf(B(X-\inf(X|B)+\delta))=\min(0,\inf(X|B)-\inf(X|B)+\delta)=0$.
Therefore $X\in\asetp$,
by property a).
Since a generic $X\in\rbg$ has been considered,
we have shown that $\rbg\subseteq\asetp$.

Conversely,
let us prove now that
\emph{if $\lpr$ is $1$-AUL and $\rbg\subseteq\asetp$, then $\lpr$ is centered}.

For this, we show that the supremum in equation \eqref{eq;lpzero} is zero,
which is equivalent to $\lpr(0|B)=0, \forall 0|B\in\dlin$.

Suppose $\rbg\subseteq\asetp$,
take $x\in\rset$,
and consider the gamble $-Bx\in\rb$.
If $x<0$ then $\inf(-Bx|B)=-x>0$,
so that $-Bx\in\rbg\subseteq\asetp$.
This implies that the supremum in equation \eqref{eq;lpzero} is at least zero.

However,
if $x>0$,
it is $\sup(-Bx)\leq 0$,
hence $-Bx\in\xsetl$.
By property b),
it follows that $-Bx\notin\asetp$ for any positive $x$ and,
therefore,
the supremum in equation \eqref{eq;lpzero} is precisely zero.
\end{proof}
Comparing Propositions \ref{pro:lower_from_set} and \ref{pro:set_from_lower} with,
respectively, Propositions \ref{pro:lower_from_setc} and \ref{pro:set_from_lowerc},
we note that, in addition to the constraint $\lambda=1$,
$2$-convexity requires no condition like c) and c')
in Propositions \ref{pro:lower_from_set} and \ref{pro:set_from_lower} respectively.
Referring, for instance, to c'), this means that,
given $X, Y\in\asetp$ with $X+Y\neq 0$,
$2$-convexity does not
guarantee $\sup(X+Y)>0$:
summing up two individually desirable gambles could therefore give rise to a partial or even to a sure loss.
Moreover, a non-centered $2$-convex $\lpr$ suffers from a more serious shortcoming:
either it is not even $1$-AUL, or
$\rbg\subseteq\asetp$ does not necessarily hold,
meaning that a non-negative gamble $X=BX$ ($X\neq 0$) exists
that is considered non-desirable.
The main drawbacks of $2$-convexity relative to $2$-coherence are therefore clearly pointed out also by a comparison through desirability axioms.

\section{Conclusions}
\label{sec:conclusions} $N$-convex and $n$-coherent conditional
lower previsions broaden the spectrum of uncertainty measures that
can be accommodated into a behavioural approach to imprecision,
including, for instance, conditional extensions of capacities and
niveloids when $n=2$. This choice for $n$ is the most neatly
distinguished from coherence, the other extreme in the spectrum, and
that retaining more interesting properties.
In particular,
centered $2$-convex and $2$-coherent previsions are \emph{stable},
meaning that they can be extended on any set preserving their consistency properties.
$2$-convex and $2$-coherent previsions also
have a clear meaning in terms of desirability.
We believe that the desirability investigation carried out in this paper,
although still at a foundational level,
is important as it displays first results on how this approach
works outside coherence and in the general conditional framework.
Further work is
necessary to investigate additional properties, like the possible
existence of envelope theorems, or properties of already defined
notions. In particular, we conjecture that the $2$-convex  or $2$-coherent natural
extensions may simplify computing the convex or the coherent natural extensions.
As a further generalisation of this work,
the consistency notions defined here could be
extended to the case of unbounded conditional random variables.
This has been done in \cite{tro14} for coherent conditional lower previsions,
while, to the best of our knowledge, a similar investigation for convex conditional previsions
is still missing.

\section*{Acknowledgements}
We wish to thank the referees for their helpful comments.

\bibliographystyle{plain}

\end{document}

%% file: macros.tex
\newcommand{\unds}{\underline{s}}
\newcommand{\Ss}{S(\unds)}

\newcommand{\LP}{\underline{P}}

\newcommand{\LE}{\underline{E}}
\newcommand{\LG}{\underline{G}}

\newcommand{\upr}{\overline{P}}
\newcommand{\lpr}{\underline{P}}

\newcommand{\LEC}{\LE_{c}}
\newcommand{\LEDC}{\LE_{2c}}
\newcommand{\LED}{\LE_{2}}
\newcommand{\LGC}{\LG_{c}}
\newcommand{\LGD}{\LG_{2}}
\newcommand{\LGDC}{\LG_{2c}}
\newcommand{\LD}{L_{2}}

\newcommand{\rset}{\mathbb{R}}
\newcommand{\natset}{\mathbb{N}}
\newcommand{\natsetp}{\mathbb{N}_{0}}

\newcommand{\lset}{\mathcal{L}}
\newcommand{\reals}{\rset}
\newcommand{\Rset}{\rset}

\newcommand{\prt}{I\dsn P}

\newcommand{\dset}{\mathcal{D}}
\newcommand{\aset}{\mathcal{A}}

\newcommand{\qset}{\mathcal{S}}

\newcommand{\xset}{\mathcal{X}}
\newcommand{\bset}{\mathcal{B}}
\newcommand{\bsetp}{\mathcal{B}^{\varnothing}}
\newcommand{\asetp}{\mathcal{A}^{\prime}}
\newcommand{\xsetg}{\mathcal{X}^{\succeq}}
\newcommand{\xsetl}{\mathcal{X}^{\preceq}}
\newcommand{\ra}{\mathcal{R}(A)}
\newcommand{\rb}{\mathcal{R}(B)}
\newcommand{\rbg}{\mathcal{R}(B)^{\succ}}
\newcommand{\rbl}{\mathcal{R}(B)^{\prec}}
\newcommand{\rbu}{\mathcal{R}(B_{1})}
\newcommand{\rbd}{\mathcal{R}(B_{2})}
\newcommand{\rbi}{\mathcal{R}(B_{i})}
\newcommand{\rbgi}{\mathcal{R}(B_{i})^{\succ}}

\newcommand{\rbud}{\mathcal{R}(B_{1}\vee B_{2})}
\newcommand{\slambda}{S(\underline{\lambda})}

\newcommand{\dsn}{\!\!}

\newcommand{\nega}[1]{\neg #1}

\newcommand{\dlin}{\dset_{LIN}}

